\documentclass[reqno]{amsart}
\usepackage{amsmath,amssymb,amsthm}
\usepackage{cite}

\numberwithin{equation}{section}

\newtheorem{Theorem}{Theorem}[section]

\newtheorem{Corollary}[Theorem]{Corollary}
\newtheorem{Lemma}[Theorem]{Lemma}

\newtheorem{Remark}[Theorem]{Remark}
\begin{document}
\title[The semilinear fractional evolution equation with damping]
{The Cauchy problem of the semilinear second order evolution equation with fractional Laplacian and damping}
\author[K. Fujiwara]{Kazumasa Fujiwara}
\address{Mathematical Institute, Tohoku University,
6-3 Aoba, Aramaki, Aoba-ku, Sendai, Miyagi, 980-8578 Japan}
\email{fujiwara.kazumasa.a6@tohoku.ac.jp}

\author[M. Ikeda]{Masahiro Ikeda}
\address{Department of Mathematics, Faculty of Science and Technology, Keio University,
3-14-1 Hiyoshi, Kohoku-ku, Yokohama, 223-8522, Japan/Center for Advanced Intelligence Project, RIKEN, Japan}
\email{masahiro.ikeda@keio.jp/masahiro.ikeda@riken.jp}

\author[Y. Wakasugi]{Yuta Wakasugi}
\address{Laboratory of Mathematics,
Graduate School of Engineering,
Hiroshima University,
Higashi-Hiroshima, 739-8527, Japan}
\email{wakasugi@hiroshima-u.ac.jp}

\begin{abstract}
In the present paper,
we prove time decay estimates of solutions in weighted Sobolev spaces to the second order evolution equation
with fractional Laplacian and damping for data in Besov spaces.
Our estimates generalize the estimates obtained in the previous studies \cite{Kar00,IIW}.
The second aim of this article is to apply these estimates
to prove small data global well-posedness for the Cauchy problem of the equation with power nonlinearities.
Especially, the estimates obtained in this paper enable us to treat more general conditions on the nonlinearities
and the spatial dimension than the results in the papers \cite{CFZ15, IIW}.
\end{abstract}



\maketitle

\tableofcontents

\section{Introduction}
\subsection{Linear Problem}
The first aim of the present paper is
to prove time decay estimates of the solution to the Cauchy problem for the second order evolution equation
with fractional Laplacian and damping:
	\begin{align}
	\begin{cases}
	\partial_t^2 u + \partial_t u + ( - \Delta)^{\sigma/2} u = 0,
	&(t,x) \in \lbrack 0,\infty) \times \mathbb R^n,\\
	u(0,x) = u_0(x),
	&x \in \mathbb R^n,\\
	\partial_t u(0,x) = u_1(x),
	&x \in \mathbb R^n,
	\end{cases}
	\label{eq:1.1}
	\end{align}
where $\sigma>0$ denotes the strength of the diffusion,
$(-\Delta)^{\sigma/2}:=\mathcal{F}^{-1}|\xi|^{\sigma}\mathcal{F}$,
$n\in \mathbb{N}$ denotes the spatial dimension,
$u=u(t,x)$ is an unknown function on $[0,\infty)\times \mathbb{R}^n$
and $(u_0,u_1)$ is an $\mathbb{R}^2$-valued prescribed function on $\mathbb{R}^n$.
Here $\mathcal{F}$ and $\mathcal{F}^{-1}$ denote the Fourier transform
and the inverse Fourier transform, respectively (see \eqref{eq:1.12} and \eqref{eq:1.13} for the definitions).

By taking the Fourier transform $\mathcal{F}$ of the problem \eqref{eq:1.1},
we see that there exists a unique global solution $u$ to the problem \eqref{eq:1.1} such that the identity
	\begin{equation}
	\label{eq:1.2}
	u(t,x)=\left\{\widetilde{\mathcal{S}_{\sigma}}(t)u_0\right\}(x)+\left\{\mathcal{S}_{\sigma}(t)u_1\right\}(x)
	\end{equation}
holds for any $(t,x)\in \lbrack 0,\infty)\times \mathbb{R}^n$.
Here $\mathcal{S}_{\sigma}(t)$ on $\lbrack 0,\infty)$ is the solution operator defined by
	\[
	\mathcal S_\sigma(t)
	:= e^{- \frac t 2} \mathcal F^{-1} L_\sigma(t,\xi) \mathcal F,
	\]
where the function $L_{\sigma}:[0,\infty)\times \mathbb{R}^n\rightarrow \mathbb{R}$ is given by
	\begin{align*}
	L_\sigma(t,\xi)
	&:= \begin{cases}
	\displaystyle
	\frac{\sinh\left(t \sqrt{ \frac 1 4-|\xi|^\sigma}\right)}{\sqrt{ \frac 1 4-|\xi|^\sigma}},
	&\mathrm{if} \quad |\xi| < 2^{- \frac 2 \sigma},\\[10pt]
	\displaystyle
	\frac{\sin\left(t \sqrt{ |\xi|^\sigma - \frac 1 4}\right)}{\sqrt{ |\xi|^\sigma - \frac 1 4}},
	&\mathrm{if} \quad |\xi| > 2^{- \frac 2 \sigma},
	\end{cases}
	\end{align*}
and the operator
$\widetilde{\mathcal{S}_{\sigma}}(t)$
is defined by
	\[
	\widetilde{ \mathcal S_\sigma}(t):=( \partial_t + 1 ) \mathcal S_\sigma(t).
	\]
The model \eqref{eq:1.1} describes various kinds of physical phenomena such as super-conductivity
(see \cite{FaLuRe10} for example).
When $\sigma=2$, the first equation of \eqref{eq:1.1} is the classical damped wave equation:
	\begin{equation}
	\label{eq:1.3}
	\partial_t^2u+\partial_tu-\Delta u=0,
	\end{equation}
which is firstly derived by Oliver Heaviside as the telegrapher's equation, and describes the current and voltage in an electrical circuit with resistance and inductance (see \cite{Gol, K79} for other physical backgrounds). 

There are many studies about asymptotic behavior in time of the solution \eqref{eq:1.2}
to the Cauchy problem \eqref{eq:1.1} as $t\rightarrow\infty$
(see \cite{Kar00, CFZ15, ChHa03, ChdLuIk13, IkNi03, RaToYo11}).
Karch \cite{Kar00} studied the asymptotic behavior of the solution to the problem \eqref{eq:1.1}
with an external force $h=h(t,x)$:
	\begin{equation}
	\label{eq:1.4}
	\partial_t^2 u + \partial_t u + ( - \Delta)^{\sigma/2} u = h.
	\end{equation}
By using the Fourier transform and its inverse again,
we find that there exists a unique global solution $u=u(t,x)$ on $[0,\infty)\times \mathbb{R}^n$
to the inhomogeneous problem \eqref{eq:1.4} with suitable initial data $(u_0,u_1)$ such that the identity
	\begin{equation}
	\label{eq:1.5}
	u(t,x)
	= \left\{\widetilde{\mathcal S_\sigma}(t) u_0\right\}(x)
	+ \left\{\mathcal S_\sigma(t) u_1\right\}(x)
	+ \int_0^t \mathcal S_\sigma(t-\tau) h(\tau,x) d \tau
	\end{equation}
holds for any $(t,x)\in \lbrack 0,\infty)\times \mathbb{R}^n$.
In Theorem 2.1 in \cite{Kar00},
Karch proved the asymptotic estimate of the solution \eqref{eq:1.5} to the problem \eqref{eq:1.4}
with data $(u_0,u_1)\in \left(L^1(\mathbb{R}^n)\right)^2$
and a suitable external force $h(t,x)\in L^1_t\left(0,\infty;L_x^1(\mathbb{R}^n)\right)$,
that is, the solution satisfies
	\begin{equation}
	\label{eq:1.6}
	\left\|u(t)-MG_{\sigma}(t)\right\|_{L^p}=o\left(t^{-\frac{n}{\sigma}\left(1-\frac{1}{p}\right)}\right)
	\end{equation}
as $t\rightarrow\infty$, where $p\in [2,\infty]$ and $M\in \mathbb{R}$ is a constant given by
	\[
	M:=\int_{\mathbb{R}^n}\{u_0(x)+u_1(x)\}dx+\int_{[0,\infty)\times\mathbb{R}^n}h(t,x)dx \, dt
	\]
and $G_{\sigma}:(0,\infty)\times\mathbb{R}^n\rightarrow\mathbb{R}_{> 0}$ with $\sigma>0$ is the heat kernel defined by
\begin{align}
\label{eq:1.7}
G_{\sigma}(t,x):=(2\pi)^{n/2}\mathcal{F}^{-1}\left[e^{-t|\xi|^{\sigma}}\right](x)=(2\pi)^{-n/2}\int_{\mathbb{R}^n}e^{-t|\xi|^{\sigma}+ix\cdot\xi}d\xi.
\end{align}
The estimate \eqref{eq:1.6} shows that solutions to the problem \eqref{eq:1.4}
with suitable data and external force
have a diffusion phenomena and behaves like the heat kernel $G_{\sigma}$ up to the constant
as $t\rightarrow\infty$ of the parabolic equation corresponding to the second order equation in \eqref{eq:1.1}:
	\[
	\partial_tu+(-\Delta)^{\sigma/2}u=0.
	\]
The estimate \eqref{eq:1.6} also leads to the sharp time decay estimate of the solution
to the problem \eqref{eq:1.4} in $L^p$-sense:
	\begin{equation}
	\label{eq:1.8}
	\|u(t)\|_{L^p}=O\left(t^{-\frac{n}{\sigma}\left(1-\frac{1}{p}\right)}\right),
	\end{equation}
as $t\rightarrow\infty$, where $p\in [2,\infty]$.
We note that the estimate \eqref{eq:1.8} with a lower exponent $p\in \lbrack 1,2)$
does not follow from Theorem 2.1 in \cite{Kar00},
thought our estimate (Corollary \ref{Corollary:2.4})) implies the estimate \eqref{eq:1.8}
with a wider range of $p\in [1,\infty]$ than above.

After the work \cite{Kar00}, Ikehata and Nishihara \cite{IkNi03},
Chill and Haraux \cite{ChHa03}, Radu, Todorova and Yordanov \cite{RaToYo11} and Ikehata, Todorova and Yordanov \cite{IkToYo13JDE}
generalized the diffusion phenomena \eqref{eq:1.6} to an abstract dissipative wave equation.

When $\sigma=2$, the first equation of \eqref{eq:1.1} is the classical damped wave equation \eqref{eq:1.3}
and the diffusion phenomena and time decay estimates in suitable function spaces were studied
in \cite{HKN04, HoOg04, IIOW, IIW, Ma76, Na04, Ni03MathZ}.

The following equation with a structural damping $(-\Delta)^{\delta/2}\partial_tu$,
which also generalizes the classical free damped wave equation \eqref{eq:1.3}, is studied:
	\begin{equation}
	\label{eq:1.9}
	\partial_t^2u+b(t)(-\Delta)^{\delta/2}\partial_tu+(-\Delta)^{\sigma/2} u=0
	\quad (t,x)\in \lbrack 0,\infty) \times \mathbb R^n,
	\end{equation}
where $\sigma>0$ and $\delta\ge 0$ are constants, and $b=b(t)$ is a time-dependent coefficient of the damping.
Due to the structural damping, solutions to \eqref{eq:1.9} have a smoothing property,
which is different from our equation \eqref{eq:1.4},
see \cite{DaEb14JDE, DaEb16, ChdLuIk13, IkNa12, FaLuRe10, KaRe15ADE, LuRe09, Wi06, Wi07}.
For other generalization of the classical damped wave equation, see \cite{dLuChIk15}.

The first aim of the present paper is to prove time decay estimates of the solution \eqref{eq:1.2}
to the problem \eqref{eq:1.1} in weighted Sobolev spaces for data in Besov spaces
(see Theorems \ref{Theorem:2.1} and \ref{Theorem:2.3}).
The estimate in weighted spaces generalizes the decay estimate \eqref{eq:1.8} derived from Theorem 2.1
in the previous study \cite{Kar00}.
Moreover the estimate in homogeneous Sobolev spaces improves the previous result (Lemma 2.1 in \cite{IIW})
(see Remark \ref{Remark:2.2}).
Another advantage of our results is that the estimate in the weighted space is useful to study nonlinear problems,
and more precisely it allows us to treat higher dimensional cases than Theorem 1.4 in \cite{CFZ15}
(see Theorem \ref{Theorem:2.7}).

\subsection{Nonlinear Problem}
The second aim of the present paper is to study well-posedness of the Cauchy problem \eqref{eq:1.1}
with a $\rho$-th order power nonlinearity $\mathcal{N}$ with $\rho>1$:
	\begin{align}
	\begin{cases}
	\partial_t^2 u + \partial_t u + ( - \Delta)^{\sigma/2} u = \mathcal N(u),
	&(t,x) \in \lbrack 0,T) \times \mathbb R^n,\\
	u(0,x) = u_0(x),
	&x \in \mathbb R^n,\\
	\partial_t u(0,x) = u_1(x),
	&x \in \mathbb R^n,
	\end{cases}
	\label{eq:1.10}
	\end{align}
where $T>0$ denotes the maximal existence time of the function $u$,
$\sigma>0$ stands for the strength of the diffusion,
$u=u(t,x)$ is an unknown function on $[0,T)\times \mathbb{R}^n$,
a pair of functions $(u_0,u_1)$ is a prescribed $\mathbb{R}^2$-valued function on $\mathbb{R}^n$,
and $\mathcal{N}=\mathcal{N}(z)$ is a nonlinear function on $\mathbb{R}$ of the order $O(|z|^{\rho})$
with a parameter $\rho>1$ (for more precise assumptions on the nonlinearity $\mathcal{N}$,
see \eqref{eq:2.7} and \eqref{eq:2.8}).

There are many mathematical results for the problem \eqref{eq:1.10} about global existence,
blow-up, asymptotic behavior of solutions, determining a critical exponent, etc.
Here the critical exponent $\rho_c$ means the number which divides global existence and blow-up, that is, if $\rho>\rho_c$, small data global existence holds, on the other hand, if $\rho<\rho_c$, small data blow-up occurs.

When $\sigma=2$, the first equation of the problem \eqref{eq:1.10} becomes the classical semilinear damped wave equation,
and small data global existence, asymptotic behavior of solutions,
and blow-up to the problem \eqref{eq:1.10} were extensively studied
(see \cite{LZ95, TY01, QZ01, HKN04, HoOg04, Ma76, Na04, Ni03MathZ}).
From their results, we see that the critical exponent $\rho_c$ to the problem \eqref{eq:1.10} with $\sigma=2$,
$\mathcal{N}(z)= |z|^{\rho}$ and data $(u_0,u_1)$ belonging to $ (L^1(\mathbb{R}^n))^2$
or the smaller space is given by the Fujita exponent $\rho_F=\rho_F(n):=1+2/n$, that is, $\rho_c=\rho_F$.
Here the Fujita exponent $\rho_F$ is originally the ($L^1$-scaling) critical exponent of the corresponding heat equation
$\partial_tu-\Delta u=|u|^{\rho}$ (see \cite{F}).
We note that the critical case $\rho=\rho_F(n)$ belongs to the blow-up region.

The critical exponent $\rho_c$ to the problem \eqref{eq:1.10} with $\sigma=2$ for data $(u_0,u_1)$
which do not belong to $(L^1(\mathbb{R}^n))^2$ in general was studied in the papers \cite{IIOW, IIW, S19}.
From the works \cite{S19, IIW},
we see that the critical exponent to the problem \eqref{eq:1.10}
with data $(u_0,u_1)$ in the weighted Sobolev space $H^{1,\alpha}(\mathbb{R}^n)\times H^{0,\alpha}(\mathbb{R}^n)$
with $\alpha\in \lbrack 0,n/2)$ is given by $1+\frac{4}{n+2\alpha}$.
We remark that in the critical case $\rho=1+\frac{4}{n+2\alpha}$ with $\alpha\in \lbrack 0,n/2)$,
the problem \eqref{eq:1.10} is globally well-posed (see \cite{IIOW} for $L^p$-framework with $p>1$).

Chen-Fan-Zhang \cite{CFZ15} studied the problem \eqref{eq:1.10}
with $\sigma>0$ and proved a small data global existence in the supercritical $\rho>1+\frac{\sigma}{n}$
and low dimensional ($n=1,2$) case.

The following semilinear equation with a structural damping $(-\Delta)^{\delta/2}\partial_tu$,
which also generalizes the semilinear classical damped wave equation, is studied:
	\begin{equation}
	\label{eq:1.11}
	\partial_t^2u+2a(-\Delta)^{\delta/2}\partial_tu+(-\Delta)^{\sigma/2} u=\mathcal{N}(u),
	\quad (t,x)\in \lbrack 0,T)\times\mathbb{R}^n,
	\end{equation}
where $a>0$ and $\delta>0$ are positive constants.
The solution to \eqref{eq:1.11} has smoothing effect due to the structural damping,
different from our problem \eqref{eq:1.10}, see \cite{DaEb14NA, DaRe14, DaoRe19, DaoRe19DCDS}.
For other related works, see \cite{KaRe18, PhKaRe15}.

Our aim of the present paper is to generalize the global well-posedness result ($n=1,2$)
to the problem \eqref{eq:1.10} obtained in \cite[Theorem 1.4]{CFZ15} to any spatial dimension $n\ge 1$.
We also improve the local well-posedness result obtained in \cite[Theorem 1.1]{IIW}
and especially we treat more general conditions on the nonlinearity $\mathcal{N}$
by using our new time decay estimates (Theorems \ref{Theorem:2.1} and \ref{Theorem:2.3}).

\subsection{Notation and basic notions}
In this subsection,
we collect the notations used throughout this paper and basic notions about the problem \eqref{eq:1.10}.

The letter $C$ indicates a generic positive constant, which may change from line to line.
For $x\in \mathbb{R}^{l}$ with $l=1$ or $n$, we use $\langle x \rangle  := \sqrt{1+|x|^2}$.
For $a \in \mathbb R$, $a_+ := \max(a,0)$, $a_- := \min(a,0)$ and $[a]$ denotes the integer part of $a$.\ 
Let $\mathcal{S}=\mathcal{S}(\mathbb{R}^n)$ be the rapidly decaying function space. For $f \in \mathcal S$,
$\mathcal F [f]$ and $\hat f$ denote the Fourier transform of $f$ defined by
	\begin{equation}
	\label{eq:1.12}
	\hat f(\xi) := \int_{\mathbb R^n} f(x) e^{- i \xi \cdot x} dx
	\end{equation}
for $\xi \in \mathbb R^n$. The inverse Fourier transform of $f$ is defined by
	\begin{equation}
	\label{eq:1.13}
	\mathcal{F}^{-1}[f](x) := (2\pi)^{-n}\int_{\mathbb R^n} f(\xi) e^{i \xi \cdot x} d\xi
	\end{equation}
for $x\in \mathbb{R}^n$.

For $s \in \mathbb R$,
we define the fractional derivative operators $(1-\Delta)^{s/2}$ and $(-\Delta)^{s/2}$
as Fourier multipliers defined by
$\mathcal F^{-1} \langle \xi \rangle^s \mathcal F$
and $\mathcal F^{-1} |\xi|^s \mathcal F$,
respectively.

For any $p\in [1,\infty]$, a mesurable set $A \subseteq \mathbb R^n$,
and a non-negative measurable function $g$,
$L^p\left(A,g(x)dx\right)$ stands for the usual Lebesgue space equipped with the norm
	\begin{align*}
	\| f \|_{L^p\left(A,g(x)dx\right)}
	&:=
	\begin{cases}
	\displaystyle
	\bigg( \int_{x\in A} |f(x)|^p g(x) dx \bigg)^{\frac 1 p},
	&\mathrm{if} \quad 1 \leq p < \infty,\\
	\mathrm{esssup}_{x \in A} \thinspace |f(x)| g(x),
	&\mathrm{if} \quad p = \infty.
	\end{cases}
	\end{align*}
For simplicity, we abbreviate $L^p(A, dx)$ to $L^p(A)$
and $L^p(\mathbb R^n)$ to $L^p$.
In addition,
for $a,b\in \mathbb{R}$ with $a<b$ and an interval $I$ between $a$ and $b$, we denote $L^p(I)$ by $L^p(a,b)$.
Similarly,
for any $p\in [1,\infty]$ and $A \subseteq \mathbb Z^n$,
$\ell^p(A)$ consists of all sequences $a=(a_j)_{j \in A}$ satisfying
that $\| a \|_{\ell^p(A)} < \infty$, where the norm $\| a \|_{\ell^p(A)}$ is defined by
	\begin{align*}
	\| a \|_{\ell^p(A)}
	&:=
	\begin{cases}
	\displaystyle
	\left( \sum_{j \in A} |a_j|^p \right)^{\frac 1 p}
	&\mathrm{if} \quad 1 \leq p < \infty,\\
	\sup_{j \in A} \thinspace |a_j|
	&\mathrm{if} \quad p = \infty.
	\end{cases}
	\end{align*}
For $s\in \mathbb{R}$ and $p\in [1,\infty]$, $H_p^s$ and $\dot H_p^s$
denote the usual inhomogeneous and homogeneous Sobolev spaces
defined by $(1-\Delta)^{- \frac s 2} L^p$ and $(-\Delta)^{- \frac s 2} L^p$,
equipped with the norms
	\[
	\| f \|_{H_p^s}
	:= \left\| (1-\Delta)^{\frac s 2} f \right\|_{L^p},
	\qquad
	\| f \|_{\dot H_p^s}
	:= \left\| (-\Delta)^{\frac s 2} f \right\|_{L^p},
	\]
respectively. We also abbreviate
$H_2^s$ and $\dot H_2^s$
to
$H^s$ and $\dot H^s$, respectively.
Similarly, for $\alpha > 0$,
$H^{0,\alpha}$ is defined by weighted Lebesgue space equipped with the norm
	\[
	\| f \|_{H^{0,\alpha}}
	:= \left\| \langle \cdot \rangle^\alpha f \right\|_{L^2}.
	\]

Next,
for $r>0$ and a Banach space $X$, let $B(r,X)$ be the open ball of $X$ with radius $r$
centered at the origin.
In particular, when $X= \mathbb{R}^n$,
we write
$B(r) = B(r, \mathbb{R}^n)$.
Let $\phi, \psi \in \mathcal S(\mathbb{R}^n)$ satisfy that
$\mathrm{supp} \thinspace \hat \psi \subset B(1)$,
$\mathrm{supp} \thinspace \hat \phi \subset B(2) \backslash B(2^{-1})$,
	\[
	\hat \psi(\xi) + \sum_{j \geq 0} \hat \phi (2^{-j} \xi)
	= 1
	\]
for any $\xi \in \mathbb Z^n$.
We set $\phi_j := \mathcal F^{-1} \hat \phi(2^{-j} \cdot)$
for any $j \in \mathbb Z$.

For $s\in \mathbb{R}$, $p\in [1,\infty]$ and $q\in[1,\infty]$,
$B_{p,q}^s=B_{p,q}^s(\mathbb{R}^n)$ denotes the inhomogeneous Besov space
defined by the collection of all $f \in \mathcal S'(\mathbb R^n)$ satisfying
	\[
	\|f\|_{B_{p,q}^s}
	:= \| \psi \ast f \|_{L^q(\mathbb R^n)}
	+ \| 2^{sj} \|\phi_j \ast f\|_{L^p} \|_{\ell^q(\mathbb Z_{\geq 0})}
	< \infty,
	\]
where $\mathbb Z_{\geq 0}$ is the collection of all non-negative integers.
$\dot B_{p,q}^s=\dot B_{p,q}^s(\mathbb{R}^n)$ denotes the homogeneous Besov space,
defined by the collection of all $f \in \mathcal S' / \mathcal P$ satisfying that
	\[
	\|f\|_{\dot B_{p,q}^s}
	:= \| 2^{sj} \|\phi_j \ast f\|_{L^p} \|_{\ell^q(\mathbb Z)}
	< \infty,
	\]
where $\mathcal P$ is the collection of all polynomials.
We note that for any $s > 0$ and $1 \leq p, q \leq \infty$,
the relation $B_{p,q}^s = L^p \cap \dot B_{p,q}^s$
and the equivalence
	\begin{align}
	\| f \|_{B_{p,q}^s}
	\sim \| f \|_{L^p} + \| f \|_{\dot B_{p,q}^s}
	\label{eq:1.14}
	\end{align}
hold, where \eqref{eq:1.14} indicates that there exist positive constants $c$ and $C$
such that for any $f \in B_{p,q}^s$, the estimates
	\[
	c \left(\| f \|_{L^p} + \| f \|_{\dot B_{p,q}^s}\right)
	\leq \| f \|_{B_{p,q}^s}
	\leq C \left(\| f \|_{L^p} + \| f \|_{\dot B_{p,q}^s}\right)
	\]
hold.
For the details,
see \cite[6.3.2 Theorem]{INT}.
It is also known that
$\dot B_{p,q}^s$ admits equivalent norms.
Let $\tau_y$ be the translation operator with $y \in \mathbb R^n$
defined by $\tau_y f = f(y+\cdot)$.
When $0 \leq s < 1$,
the equivalence
	\begin{align}
	\| f \|_{\dot B_{p,q}^s}
	&\sim
	\Big\| \rho^{-s} \sup_{|y| \leq \rho}
		\| (\tau_y - 1) f \|_{L^p(\mathbb R^n)}
	\Big\|_{L^q(0,\infty,\frac {d \rho} \rho)}
	\nonumber\\
	&= \bigg( \int_0^\infty \rho^{-sq} \sup_{|y| \leq \rho}
	\| (\tau_y - 1)f \|_{L^p}^q \frac {d \rho}{\rho} \bigg)^{\frac 1 q}
	\label{eq:1.15}
	\end{align}
holds for any $1 \leq p,q \leq \infty$.
More generally, for positive integer $m$ and $0 \leq s < m$,
the equivalence
	\begin{align}
	\| f \|_{\dot B_{p,q}^s}
	\sim \Big\| \rho^{-s} \sup_{|y| \leq \rho}
		\| (\tau_{y/2} - \tau_{-y/2})^m f \|_{L^p}
	\Big\|_{L^q\left(0,\infty,\frac {d \rho} \rho\right)}
	\label{eq:1.16}
	\end{align}
holds for any $1 \leq p,q \leq \infty$.
For the details,
see \cite[6.3.1 Theorem]{INT}.

We introduce the nonlinear mapping $\Phi$ on a suitable function space given by
	\[
	\Phi(u)(t)
	:= \widetilde{\mathcal S_\sigma}(t) u_0
	+ \mathcal S_\sigma(t) u_1
	+ \int_0^t \mathcal S_\sigma(t-\tau) \mathcal N(u(\tau)) d \tau.
	\]
Then the integral form of the problem \eqref{eq:1.10} is given by
	\begin{align}
	u(t)
	= \Phi(u)(t).
	\label{eq:1.17}
	\end{align}
Let $T>0$, $s\in \mathbb{R}$ and $\alpha\in \mathbb{R}$.
We say that a function $u=u(t,x):[0,T]\times\mathbb{R}^n\rightarrow\mathbb{R}$
is a mild solution of \eqref{eq:1.10} if $(u,\partial_tu)$ belongs to
	\[
	C\left([0,T]; (H^s \cap H^{0,\alpha}) \times (H^{s-\sigma/2} \cap H^{0,\alpha}) \right)
	\]
and it satisfies the integral form \eqref{eq:1.17} in $(H^s \cap H^{0,\alpha})$-sense.

\section{Main Results}
In this section, we state our main results in the present paper.

\subsection{Linear Problem}
In this subsection,
we state time decay estimates for the free solution
in weighted Sobolev spaces to the Cauchy problem \eqref{eq:1.1}.

The following theorem means time decay estimates of the solution $\mathcal{S}_{\sigma}(t)g$
to the problem \eqref{eq:1.1} for the initial data $u_0=0$ and $u_1=g$:
\begin{Theorem}[Time decay estimates for the solution $\mathcal S_\sigma(t)g$]
\label{Theorem:2.1}
Let $n\in \mathbb{N}$, $\sigma>0$, $\gamma \in [1,2]$, and $s_1,s_2\in \mathbb{R}$ with $s_1 \geq s_2 \geq 0$.
Then there exists a positive constant $C=C(n,\sigma,\gamma,s_1,s_2)>0$ such that for any $t\in \lbrack 0,\infty)$,
the estimate
	\begin{align}
	\left\| D^{s_1} \mathcal S_\sigma (t) g \right\|_{L^2}
	&\leq C
	\langle t \rangle^{- \frac 1 \sigma (s_1 - \frac n 2)}
	\min \big( \langle t \rangle^{\frac 1 \sigma (s_2 - \frac n \gamma)} \| g \|_{\dot B_{\gamma,2}^{s_2}},
	\langle t \rangle^{\frac 1 \sigma (s_2-n)} \| g \|_{\dot H_1^{s_2}}\big)
	\nonumber\\
	&+ C e^{- \frac t 4}
	\left\| D^{s_1} (1-\Delta)^{-\frac \sigma 4}
		\mathcal F^{-1} ( ( 1 - \hat \psi(2^{- [ \frac 2 \sigma + 1]} \cdot )) \hat g) \right\|_{L^2},
	\label{eq:2.1}
	\end{align}
holds provided that the right hand side is finite.

Moreover, let $\nu \in [1,2]$ and $\beta>0$ with
	\begin{align}
	\beta <
	\begin{cases}
		\infty
		&\mathrm{if} \quad \sigma \in 2\mathbb{N},\\
		\frac{n}{\gamma} - \frac{n}{2} + \sigma
		&\mathrm{otherwise}.
	\end{cases}
	\label{eq:2.2}
	\end{align}
Then there exists a positive constant $C=C(n,\sigma,\gamma,\nu,\beta)>0$ such that for any $t\in \lbrack 0,\infty)$,
the estimate
	\begin{align}
	\left\| | \cdot |^\beta \mathcal S_\sigma (t) g \right\|_{L^2}
	&\leq C
	\langle t \rangle^{- \frac 1 \sigma( - \beta - \frac n 2 ) + \frac 1 \sigma (- \frac n \gamma) }
	\| g \|_{L^\gamma}
	+ C \langle t \rangle^{-\frac 1 \sigma (- \frac n 2) + \frac 1 \sigma ( - \frac n \nu)}
	\left\| |\cdot|^\beta g \right\|_{L^\nu}
	\nonumber\\
	&+ C e^{- \frac t 4}
	\| \langle \cdot \rangle^\beta g \|_{L^{q_\sigma}}
	\label{eq:2.3}
	\end{align}
holds provided that the right hand side is finite.
\end{Theorem}
The proof of Theorem \ref{Theorem:2.1}
is based on the argument of the proof of Lemma 2.1 in \cite{IIW}
(See also Lemma 2.1 in \cite{HKN04} and Theorem 1.1 in \cite{IIOW} for the similar argument).

\begin{Remark}
\label{Remark:2.2}
Estimate \eqref{eq:2.1} is sharper than the estimate (2.1) in Lemma 2.1 in \cite{IIW}
when $\gamma > 1$.
Indeed, this follows from the fact that relation $\dot H_{\gamma}^{s_2} \hookrightarrow \dot B_{\gamma,2}^{s_2}$ holds
for any $s_2 \geq 0$ and $\gamma \in (1,2]$.
For the detail, see \cite[6.4.4 Theorem]{INT} for example.
\end{Remark}

The following theorem means time decay estimates in a weighted Sobolev space of the solution
$\widetilde{\mathcal{S}_{\sigma}(t)}f$ to the problem \eqref{eq:1.1}
with the initial data $u_0=f$ and $u_1=0$:
\begin{Theorem}[Time decay estimates for the solution $\widetilde{\mathcal{S}_{\sigma}(t)}f$]
\label{Theorem:2.3}
Let $n\in \mathbb{N}$, $\sigma>0$, $s_1,s_2\in \mathbb{R}$ with $s_1 \geq s_2 \geq 0$ and $\gamma\in [1,2]$.
Then there exists a positive constant $C=C(n,\sigma,\gamma,s_1,s_2)>0$ such that for any $t\in \lbrack 0,\infty)$,
the estimate
	\begin{align*}
	\left\| D^{s_1} \widetilde{\mathcal S_\sigma} (t) f \right\|_{L^2}
	&\leq C
	\langle t \rangle^{- \frac 1 \sigma (s_1 - \frac n 2)}
	\min\big(\langle t \rangle^{\frac 1 \sigma (s_2 - \frac n \gamma)}
	\| f \|_{\dot B_{\gamma,2}^{s_2}},
	\langle t \rangle^{\frac 1 \sigma(s_2-n)} \| f \|_{\dot H_1^{s_2}} \big)\\
	&+ C e^{- \frac t 4}
	\| D^{s_1} \mathcal F^{-1} ( ( 1 - \hat \psi(2^{- [ \frac 2 \sigma + 1]} \cdot ) ) \hat f) \|_{L^2}
	\end{align*}
holds provided that the right hand side is finite.

Moreover, let $\nu \in [1,2]$ and $\beta>0$ satisfy the condition \eqref{eq:2.2}.
Then there exists a positive constant $C=C(n,\sigma,\gamma,\nu,\beta)>0$ such that for any $t\in \lbrack 0,\infty)$,
the estimate
	\begin{align*}
	\left\| | \cdot |^\beta \widetilde{\mathcal S_\sigma} (t) f \right\|_{L^2}
	&\leq C
	\langle t \rangle^{- \frac 1 \sigma( - \beta - \frac n 2 ) + \frac 1 \sigma (- \frac n \gamma) }
	\| f \|_{L^\gamma}
	+ C \langle t \rangle^{-\frac 1 \sigma (- \frac n 2) + \frac 1 \sigma ( - \frac n \nu)}
	\left\| |\cdot|^\beta f \right\|_{L^\nu}
	\nonumber\\
	&+ C e^{- \frac t 4}
	\left\| \langle \cdot \rangle^\beta f \right\|_{L^2}.
	\end{align*}
holds.
\end{Theorem}

From Theorems \ref{Theorem:2.1} and \ref{Theorem:2.3} and the argument of the proof \cite[Lemma2.3]{IIW},
we can derive the following time decay estimates of $L^p$-norm for the solutions to the linear problem \eqref{eq:1.1}:
\begin{Corollary}[Time decay estimates in $L^p$ for the solutions]
\label{Corollary:2.4}
Let $\sigma>0$, $p\in [1,\infty]$, $n\in\mathbb{N}$, $s > \frac{n}{2}$ and $\beta > \frac{n}{2}$.
Then there exists a positive constant $C=C(p,n,s,\beta)>0$ such that for any $t\in \lbrack 0,\infty)$,
the estimates
	\begin{align*}
	&\| \mathcal{S}_{\sigma}(t) g \|_{L^p}
	\le C (1+t)^{-\frac{n}{\sigma}\left(1-\frac{1}{p}\right)}
	( \| g \|_{\dot{H}^{s-\frac{\sigma}{2}}}
	+\| \langle \cdot \rangle^{\beta} g \|_{L^2}),\\
	&\| \widetilde{\mathcal S_\sigma} (t) f \|_{L^p}
	\le C (1+t)^{-\frac{n}{\sigma}\left(1-\frac{1}{p}\right)}
	( \| f \|_{\dot{H}^{s}}
	+\| \langle \cdot \rangle^{\beta} f \|_{L^2}),\\
	&\|\widetilde{\mathcal S_\sigma}(t) u_0
	+ \mathcal S_\sigma(t) u_1\|_{L^p}\\
	&\le C(1+t)^{-\frac{n}{\sigma}\left(1-\frac{1}{p}\right)}
	\left(
	\|(u_0,u_1)\|_{ \dot{H}^{s}\times \dot{H}^{s-\frac{\sigma}{2}}}
	+\|\langle \cdot\rangle^{\beta}(u_0,u_1)\|_{L^2\times L^2}\right)
	\end{align*}
hold provided that the right hand sides are finite.
\end{Corollary}

This corollary generalizes the range of the exponent $p$ for the estimate \eqref{eq:1.8}
obtained by Theorem 2.1 in \cite{Kar00} to the wider range $p\in [1,\infty]$.

\subsection{Nonlinear Problem}
In this subsection, we state our well-posedness results to the nonlinear problem \eqref{eq:1.10}. The following theorem means large data local well-posedness for arbitrary data in a weighted Sobolev space.
\begin{Theorem}[Local well-posedness]
\label{Theorem:2.5}
Let $n\in \mathbb{N}$, $\sigma>0$, $s \geq 0$ and $\alpha \ge 0$ with
	\begin{align}
	\alpha
	<
	\begin{cases}
	\infty
	&\mathrm{if} \quad \sigma \in 2\mathbb{N},\\
	\frac n 2 + \sigma
	&\mathrm{otherwise}.
	\end{cases}
	\label{eq:2.4}
	\end{align}
We assume that $\rho>\max(1,s)$ satisfies
	\begin{align}
	\rho >
	1 + \frac{\min(n,\sigma)}{2\alpha + n}
	\label{eq:2.5}
	\end{align}
and
	\begin{align}
	\rho
	\begin{cases}
	\displaystyle
	< 1 + \frac{\min(n,\sigma)}{(n-2s)_+}
	&\mathrm{if} \ ``s \not\in \mathbb Z \ \mathrm{and} \ \rho-[s] < 2''
	\ \mathrm{or} \
	``s \in \mathbb Z \ \mathrm{and} \ \rho-s \leq 1''
	,\\[5pt]
	\displaystyle
	\leq 1 + \frac{\min(n,\sigma)}{(n-2s)_+}
	&\mathrm{otherwise},
	\end{cases}
	\label{eq:2.6}
	\end{align}
where $1/0$ is regarded as $\infty$.
Let $\overline{s}\in \left[0, \max(1,[s]-1)\right]$ and $\mathcal N \in C^{\max([s]-1,1)}(\mathbb{R})$ be a function satisfying the estimates
	\begin{align}
	\mathcal N^{(\overline s)}(0) & = 0, \label{eq:2.7}
	\\
	\left|\mathcal N^{(\overline s)}(z_1) - \mathcal N^{(\overline s)}(z_2)\right|
	&\leq C
	\begin{cases}
	(|z_1|+|z_2|)^{\rho-\overline s - 1} |z_1-z_2|,
	&\mathrm{if} \quad \rho -\overline s \geq 1,\\
	|z_1-z_2|^{\rho-\overline s},
	&\mathrm{if} \quad \rho -\overline s \leq 1
	\end{cases}\label{eq:2.8}
	\end{align}
for any $z_1,z_2\in \mathbb{R}$.
Then for any $r>0$ and the initial data 
	\[
	(u_0,u_1)
	\in B(r, (H^s \cap H^{0,\alpha}) \times (H^{s-\frac{\sigma}{2}} \cap H^{0,\alpha}))
	\]
there exists $T=T(r) > 0$ such that
the Cauchy problem \eqref{eq:1.10} admits
a unique local mild solution
	\[
	(u,\partial_tu)
	\in C( \lbrack 0,T); (H^s \cap H^{0,\alpha} ) \times (H^{s-\frac{\sigma}{2}} \cap H^{0,\alpha} ))
	\]
depending continuously on initial data in
$(H^s \cap H^{0,\alpha} )
	\times (H^{s-\frac{\sigma}{2}} \cap H^{0,\alpha} )$.
Moreover, let $T_m$ be the maximal existence time of the solution $u$. Then if $T_m<\infty$,
then $u$ satisfies that
	\[
	\liminf_{t \to T_m-0}\left( \| u(t) \|_{H^{s} \cap H^{0,\alpha}} + \| \partial_tu(t) \|_{H^{s-\frac{\sigma}{2}} \cap H^{0,\alpha}}\right)= \infty.
	\]
\end{Theorem}

\begin{Remark}
The above local well-posedness result generalizes that of the
previous study \cite{IIW} (Theorem 1.1 in \cite{IIW}) to
the case of fractional Laplacian.
Not only that, it relaxes the range of the exponent $\rho$ of Theorem 1.1 in \cite{IIW},
in which
$\rho > [s] + 1$ was assumed (see Lemma \ref{Lemma:6.3}).
\end{Remark}


The following theorem means small data global well-posedness
in a weighted Sobolev space $H^{s} \cap H^{0,\alpha}$
to the problem \eqref{eq:1.10} and time decay estimates of the global solution
in the supercritical or critical case,
i.e. $\rho>1+\frac{\sigma}{n}$ with $\alpha\ge \frac{n}{2}$
or $\rho\ge 1+\frac{2\sigma}{2\alpha+n}$ with $\alpha\in \lbrack 0,\frac{n}{2})$:

\begin{Theorem}[Global well-posedness in the critical or supercritical case]
\label{Theorem:2.7}
Besides the assumptions in Theorem \ref{Theorem:2.5}, we assume that the exponent $\rho$ satisfies
	\begin{align*}
	\rho
	\begin{cases}
	\displaystyle
	\ge 1+ \frac{2\sigma}{2\alpha+n},
	&\text{if}\ \alpha\in \big\lbrack 0,\frac{n}{2}\big),\\[5pt]
	\displaystyle
	> 1+ \frac{\sigma}{n},
	&\text{if}\ \alpha\in \big\lbrack \frac{n}{2},\infty \big).
	\end{cases}
	\end{align*}
Then there exists $\varepsilon=\varepsilon(n,\sigma,\rho,\alpha,s)>0$ such that for any initial data
	\[
	(u_0,u_1)\in B(\varepsilon, (H^s \cap H^{0,\alpha}) \times (H^{s-\frac{\sigma}{2}} \cap H^{0,\alpha}) ),
	\]
the mild local solution $u$ to the Cauchy problem \eqref{eq:1.10} with the data $(u_0,u_1)$ obtained in Theorem \ref{Theorem:2.5} is extended globally, namely, $T_m=\infty$. Moreover, for any $r\ge 1$ with
	\[
	r
	\begin{cases}
	= 1
	&\mathrm{if} \quad \alpha\in (\frac{n}{2},\infty),\\
	> \frac{2n}{2\alpha + n}
	&\mathrm{if} \quad \alpha \in [0,\frac{n}{2}],
	\end{cases}
	\]
there exists a positive constant $C=C(n,\sigma,\rho,\alpha,s,r)>0$ such that for any $t\in [0,\infty)$, the following estimates hold:
    \begin{align}
    \label{eq:2.9}
    \|u(t)\|_{L^2}
    &\leq C\varepsilon \langle t \rangle^{\frac{n}{\sigma} ( \frac 1 2 - \frac 1 r)},\\
    \label{eq:2.10}
    \|u(t)\|_{\dot H^s}
    &\leq C\varepsilon \langle t \rangle^{ - \frac s \sigma + \frac{n}{\sigma} ( \frac 1 2 - \frac 1 r)},\\
    \|| \cdot |^\alpha u(t)\|_{L^2}
    &\leq C\varepsilon \langle t \rangle^{ \frac \alpha \sigma + \frac{n}{\sigma} ( \frac 1 2 - \frac 1 r)}.
    \label{eq:2.11}
    \end{align}
\end{Theorem}

\section{Proof of Theorems \ref{Theorem:2.1} and \ref{Theorem:2.3}}

In this section,
we prove time decay estimates of the free solution to the problem  (\ref{eq:1.1}) (Theorems \ref{Theorem:2.1} and \ref{Theorem:2.3} and Corollary \ref{Corollary:2.4}).
We especially show a pointwise estimate (Lemma \ref{Lemma:3.2}) for the low frequency part of the free solution
by using the homogeneous Besov norms (see \ref{eq:1.15}-\ref{eq:1.16}).
The idea to apply homogeneous Besov norms goes back at least as far as \cite{SS}.
The lemma below gives a sufficient condition of propagators
to apply their pointwise estimates.

\begin{Lemma}
\label{Lemma:3.1}
Let $m\in \mathbb{N}$ with $m \geq 1$ and $s\in (0,2m)$.
Let $f \in C(\mathbb R^n) \cap C^{2m}(\mathbb R^n \backslash \{0\})$.
If there exists a constant $C_0$ such that the estimates
	\begin{align}
	|f(\xi)|
	&\leq C_0 |\xi|^{s-n},
	\label{eq:3.1}\\
	\max_{|\alpha|=2m} |\partial^\alpha f(\xi)|
	&\leq C_0 |\xi|^{s-n-2m}
	\label{eq:3.2}
	\end{align}
hold for any $\xi \neq 0$,
then $f$ belongs to $\dot B_{1,\infty}^s(\mathbb{R}^n)$ with the estimate $\|f\|_{\dot B_{1,\infty}^s}\le C_1$, where $C_1$ depends only on $s$, $m$ and $C_0$.
\end{Lemma}

\begin{proof}
It is enough to show that the estimate
	\[
	\|(\tau_{\eta/2} - \tau_{-\eta/2})^{2m}f\|_{L^1}
	\leq C_1 |\eta|^s
	\]
holds for any $\eta\in \mathbb{R}^n$.
Indeed, \eqref{eq:1.16} and the estimate above imply
that $f \in \dot B_{1,\infty}^s(\mathbb{R}^n)$ with the estimate $\|f\|_{\dot B_{1,\infty}^s}\le C_1$.

At first,
\eqref{eq:3.1} and the positivity of $s$ imply that
the estimates
	\begin{align*}
	&\int_{|\xi| < 2m |\eta|}
	|(\tau_{\eta/2} - \tau_{-\eta/2})^{2m}f(\xi)| d \xi\\
	&\leq C \int_{|\xi|< 3m |\eta|} |f(\xi)| d \xi
	\leq C \int_{|\xi|< 3m |\eta|} |\xi|^{s-n} d \xi
	\leq C |\eta|^s
	\end{align*}
hold.
Moreover
the identities
	\begin{align*}
	(\tau_{\eta/2} - \tau_{-\eta/2})^{2m}f(\xi)
	&= (\tau_{\eta/2} - \tau_{-\eta/2})^{2m-1} \int_{-1/2}^{1/2} \nabla f(\xi+ \theta_1 \eta) d \theta_1 \cdot \eta\\
	&= \sum_{|\alpha|=2m}
	\int_{\theta \in [-1/2,1/2]^{2m}} \partial^\alpha f\bigg( \xi+ \eta \sum_{j=1}^{2m} \theta_j \bigg)
	d \theta \eta^\alpha
	\end{align*}
follow from the fundamental theorem of calculus.
In the case of $|\xi|/2m > |\eta|$, the assumption \eqref{eq:3.2} implies that the following estimates hold
for any multi-indices $\alpha$ with $|\alpha|=2m$:
	\begin{align*}
	\bigg|\partial^\alpha f \bigg(\xi+ \eta \sum_{j=1}^{2m} \theta_j\bigg)\bigg|
	&\leq C \bigg|\xi + \eta \sum_{j=1}^{2m} \theta_j \bigg|^{s-n-2m}\\
	&\leq C (|\xi| - m |\eta|)^{s-n-2m}\\
	&\leq C |\xi|^{s-n-2m}.
	\end{align*}
Therefore, we have
	\begin{align*}
	\int_{|\xi|/2m \geq |\eta|}
	|(\tau_{\eta/2} - \tau_{-\eta/2})^{2m} f(\xi)| d \xi
	\leq C |\eta|^{2m} \int_{|\xi|/2m > |\eta|} |\xi|^{s-n-2m} d \xi
	\leq C |\eta|^s,
	\end{align*}
where we have used the assumption $s < 2m$.
\end{proof}

Now we show the pointwise estimate of
the low frequency part of the free solutions.

\begin{Lemma}
\label{Lemma:3.2}
Let $\sigma > 0$. Then there exists a positive constant $C=C(\sigma)>0$ such that the estimate holds:
	\begin{align}
	\left| \mathcal F^{-1} \left( e^{- \frac t 2} L_\sigma(t) \hat \psi(2^{\frac 2 \sigma} \cdot )\right )(x) \right|
	\leq C
	\begin{cases}
		t \langle x \rangle^{-n-\sigma}
		& \mathrm{if} \quad t \leq 1,\\
		t^{-n/\sigma}
		\langle t^{-1/\sigma} x \rangle^{-n-\sigma}
		& \mathrm{if} \quad t \geq 1.
	\end{cases}
	\label{eq:3.3}
	\end{align}
Moreover, for $\sigma\in 2\mathbb{N}$,
for any positive number $j$,
there exists a positive constant $C_j$ such that
the following estimate holds:
	\begin{align}
	\left| \mathcal F^{-1} \left( e^{- \frac t 2} L_\sigma(t) \hat \psi(2^{\frac 2 \sigma} \cdot )\right )(x) \right|
	\leq C_j
	\begin{cases}
		t \langle x \rangle^{-j}
		& \mathrm{if} \quad t \leq 1,\\
		t^{-n/\sigma}
		\langle t^{-1/\sigma} x \rangle^{-j}
		& \mathrm{if} \quad t \geq 1.
	\end{cases}
	\label{eq:3.4}
	\end{align}
\end{Lemma}

\begin{Remark}
Lemma \ref{Lemma:3.2} implies that
$\mathcal F^{-1} (e^{-\frac t 2} L_\sigma(t) \hat \psi(2^{\frac 2 \sigma} \cdot))$
is controlled by the heat kernel $G_\sigma(t)$ with $\sigma < 2$, which is defined by (\ref{eq:1.7}), 
pointwisely with some constants.
For estimates of $G_\sigma$,
we refer the reader \cite{V} and references therein.
\end{Remark}

\begin{Remark}
In Proposition 7 in \cite{CFZ14},
a similar pointwise estimete as \eqref{eq:3.3} is stated.
However, we give a proof of it for completeness of the paper. 
\end{Remark}

\begin{proof}[Proof of Lemma \ref{Lemma:3.2}]
We employ the approach for the proof of \cite[Theorem 2.2]{SS}
with dividing the proof into two cases where $t \leq 1$ and $t \geq 1$.

\paragraph{Case 1: $t \leq 1$}
By the Hausdorff-Young inequality,
we have
	\begin{align}
	&t^{-1} \langle x \rangle^{\sigma+n}
	\left| \mathcal F^{-1} \left(e^{- \frac t 2} L_\sigma(t) \hat \psi(2^{\frac 2 \sigma} \cdot)\right) \right|
	\nonumber\\
	&\leq t^{-1}
	\left| \mathcal F^{-1} (e^{- \frac t 2} L_\sigma(t) \hat \psi(2^{\frac 2 \sigma} \cdot)) \right|
	+ t^{-1} \langle x \rangle^{\sigma}
	\sum_{\ell=1}^n \left| x_\ell^n \mathcal F^{-1} (e^{- \frac t 2} L_\sigma(t) \hat \psi(2^{\frac 2 \sigma} \cdot)) \right|
	\nonumber\\
	&\leq C
	\left\| t^{-1} e^{- \frac t 2} L_\sigma(t) \hat \psi(2^{\frac 2 \sigma} \cdot) \right\|_{L^1}
	+ \sum_{\ell=1}^n
	\left\|\partial_\ell^n t^{-1} e^{- \frac t 2} L_\sigma(t) \hat \psi(2^{\frac 2 \sigma} \cdot)
	\right\|_{B_{1,\infty}^{\sigma}}
	\label{eq:3.5}
	\end{align}
for any $t > 0$ and $x \in \mathbb R^n$.
Next, we note that
the estimate
	\begin{align}
	\bigg| \frac{d^k}{d \vartheta^k} \frac{\sinh(t \sqrt{\vartheta})}{t \sqrt{\vartheta}} \bigg|
	\leq 2^{1-k} e^{t \sqrt \vartheta}
	\label{eq:3.6}
	\end{align}
holds for any $\vartheta \geq 0$, $0 < t < 1$, and $k \in \mathbb Z_{\geq 0}$.
Indeed,
the identities
	\begin{align*}
	\frac{\sinh(t \sqrt \vartheta)}{t \sqrt \vartheta}
	&= \int_0^1 \cosh(\theta_0 t \sqrt \vartheta) d \theta_0\\
	\frac{d}{d \vartheta} \frac{\sinh(t \sqrt \vartheta)}{t \sqrt \vartheta}
	&= 2^{-1} t \int_0^1 \theta_0 \frac{\sinh(\theta_0 t \sqrt \vartheta)}{\sqrt \vartheta} d \theta_0
	= 2^{-1} t^2 \int_0^1 \int_0^1 \theta_0^2 \cosh(\theta_0 \theta_1 t \sqrt \vartheta) d \theta_1 d \theta_0
	\end{align*}
hold.
Then it may be shown inductively that
the identity
	\[
	\frac{d^k}{d \vartheta^k} \frac{\sinh(t \sqrt \vartheta)}{t \sqrt \vartheta}
	=
	2^{-k} t^{2k} \int_{\theta \in [0,1]^{k+1}}
	\prod_{j=0}^k \theta_j^{2(k-j)}
	\cosh\bigg(t \sqrt \vartheta \prod_{j=0}^k \theta_j \bigg) d \theta
	\]
holds for any $k \in \mathbb Z_{\geq 0}$,
which implies \eqref{eq:3.6}.
Therefore, for $k \in \mathbb Z_{\geq 0}$,
there exists a constant $C_k$ depending only on $k$ such that
	\[
	\sup_{0 < t \leq 1}
	\left| \partial_\ell^k (t^{-1} e^{-\frac t 2} L_\sigma(t,\xi) \hat \psi(2^{\frac 2 \sigma} \xi))\right|
	\leq C_k |\xi|^{\sigma-k} \chi_{B\left(2^{-\frac 2 \sigma}\right)}.
	\]
This and Lemma \ref{Lemma:3.1} imply
that the RHS of \eqref{eq:3.5} is uniformly bounded with respect to $t\in [0,1]$.
Therefore, the first estimate of \eqref{eq:3.3} hold.
When $\sigma$ is even, \eqref{eq:3.4} follows similarly from the fact that
the estimate
	\[
	\sup_{0 < t < 1}
	\left\| \partial_\ell^k (t^{-1} e^{-\frac t 2} L_\sigma(t) \hat \psi(2^{\frac 2 \sigma} \cdot))\right\|_{L^1}
	< \infty
	\]
holds for any non-negative integer $k$.

\bigskip

\paragraph{Case 2: $t \geq 1$}
We set
	\[
	\Lambda_\sigma(t,\eta)
	:= e^{-t/2} L_\sigma\left(t,t^{- \frac 1 \sigma}\eta\right) \hat \psi\left(2^{\frac 2 \sigma} t^{- \frac 1 \sigma}\eta\right).
	\]
Then it is computed directly that
	\begin{align*}
	&\mathcal F^{-1} \left(e^{-t/2} L_\sigma(t,\cdot) \hat \psi(2^{\frac 2 \sigma} \cdot) \right)(x)\\
	&= (2 \pi)^{-n} \int_{\mathbb R^n} e^{-t/2} L_\sigma(t,\xi) \hat \psi( 2^{\frac 2 \sigma} \xi)
	e^{i \xi \cdot x} d \xi \\
	&= (2 \pi)^{-n} t^{-n/\sigma}
	\int_{\mathbb R^n} e^{-t/2}
		L_\sigma(t,t^{-1/\sigma}\eta) \hat \psi( 2^{\frac 2 \sigma} t^{-1/\sigma}\eta)
		e^{i \eta \cdot (t^{-1/\sigma} x)} d \eta\\
	&= t^{-n/\sigma} \mathcal F^{-1} (\Lambda_\sigma(t)) \left(t^{-1/\sigma} x\right),
	\end{align*}
where we have used changing variables with $\xi := t^{-1/\sigma} \eta$.
Then as well as \eqref{eq:3.5},
by rewriting $t^{- \frac 1 \sigma} x$ as $y$,
we have
	\begin{align}
	& \left|t^{n/\sigma} \langle t^{-1/\sigma} x \rangle^{n+\sigma}
	\mathcal F^{-1} (e^{-t/2} L_\sigma(t,\cdot) \hat \psi(2^{\frac 2 \sigma} \cdot) )(t^{-1/\sigma} x)\right|
	\nonumber\\
	&\leq \langle y \rangle^{n+\sigma}
	| \mathcal F^{-1} (\Lambda_\sigma(t)) (y)|
	\nonumber\\
	&\leq
	| \mathcal F^{-1} (\Lambda_\sigma(t)) (y)|
	+ C \sum_{\ell=1}^n \langle y \rangle^{\sigma} | y_\ell^n \mathcal F^{-1} (\Lambda_\sigma(t)) (y)|
	\nonumber\\
	&\leq \| \Lambda_{\sigma}(t)\|_{L^1}
	+ C \sum_{\ell=1}^n \| \partial_\ell^n \Lambda_{\sigma}(t)\|_{B_{1,\infty}^\sigma}.
	\label{eq:3.7}
	\end{align}
We note that the estimates $\frac{|\eta|^\sigma} t \leq \frac 1 4$ and
	\[
	\frac{t}{2} - t \sqrt{ \frac 1 4 - \frac{|\eta|^\sigma} t }
	= \frac{|\eta|^\sigma}{ \frac 1 2 + \sqrt{ \frac1 4 - \frac{|\eta|^\sigma} t}}
	\geq |\eta|^\sigma
	\]
hold. Moreover, since $\mathrm{supp} \thinspace \psi \subset B(1)$,
there exists $\delta > 0$ such that
	\[
	\sqrt{ \frac1 4 - \frac{|\eta|^\sigma} t}
	> \delta.
	\]
The estimates above imply that
	\begin{align}
	\sup_{t \geq 1} \left|\partial_\ell^k \Lambda_\sigma(t,\eta)\right|
	&\leq C_k \sup_{t \geq 1} \left(|\eta|^{\sigma-k} + |\eta|^{k \sigma-k}\right)
	\exp\bigg(-\frac{t}{2} + t \sqrt{ \frac 1 4 - \frac{|\eta|^\sigma} t } \bigg)
	\nonumber\\
	&\leq C_k |\eta|^{\sigma-k} e^{-|\eta|^\sigma/2}
	\label{eq:3.8}
	\end{align}
for any $k \in \mathbb Z_{\geq 0}$ with a constant $C_k$ depending only on $k$.
Therefore, the first estimate of \eqref{eq:3.3} hold
because \eqref{eq:3.8} and Lemma \ref{Lemma:3.1} imply
that the RHS of the last estimate of \eqref{eq:3.7}
is uniformly bounded with respect to $t \geq 1$.
When $\sigma$ is even, \eqref{eq:3.4} follows similarly from the fact that
the estimate
	\[
	\sup_{0 < t < 1}
	\left\| \partial_\ell^k \Lambda_\sigma(t) \right\|_{L^1}
	< \infty
	\]
holds for any non-negative integer $k$, which completes the proof of the lemma.

\end{proof}

Now we give a proof of Theorem \ref{Theorem:2.1}.
\begin{proof}[Proof of Theorem \ref{Theorem:2.1}]
First we prove the estimate \eqref{eq:2.1}.
We note that for any $\sigma>0$, $s \geq 0$ and $q \geq 1$, there exists a positive constant $C=C(\sigma,s,q)>0$ such that the estimates
	\begin{align}
	e^{- \frac t 2} \| |\xi|^s L_\sigma(t,\xi) \|_{L^q\left(|\xi| < 2^{-\frac 2 \sigma}\right)}
	&\leq C
	\| |\xi|^s e^{-t |\xi|^\sigma} \|_{L^q\left(|\xi| < 2^{-\frac 2 \sigma}\right)}
	\nonumber\\
	&\leq C
	\begin{cases}
	\| |\xi|^s \|_{L^q\left(|\xi| < 2^{-\frac 2 \sigma}\right)}
	&\mathrm{if} \quad t \leq 1,\\
	t^{-\frac s \sigma - \frac n {\sigma q}}
	\| |\xi|^s e^{- |\xi|^\sigma} \|_{L^q}
	&\mathrm{if} \quad t \geq 1
	\end{cases}
	\label{eq:3.9}
	\end{align}
and
	\begin{align}
	\sup_{t > 0} \left\| \langle \xi \rangle^{\frac \sigma 2} L_\sigma(t,\xi)
	\right\|_{L^\infty\left(|\xi| > 2^{-\frac 2 \sigma}\right)}
	< \infty
	\label{eq:3.10}
	\end{align}
hold.
Then \eqref{eq:3.9} implies that for $j\in \mathbb{Z}$ with $j \leq - \frac 2 \sigma - 1$,
we have
	\begin{align}
	\| D^{s_1} \phi_j \ast \mathcal S_\sigma(t) f \|_{L^2}
	&= \left\| e^{-\frac t 2} L_\sigma(t) |\xi|^{s_1-s_2} |\xi|^{s_2} \hat \phi_j \hat f \right\|_{L^2}
	\nonumber\\
	&\leq \left\| e^{-\frac t 2} L_\sigma(t) |\xi|^{s_1-s_2}
	\right\|_{L^{\frac{2\gamma'}{\gamma'-2}}\left(|\xi| < 2^{-\frac 2 \sigma}\right)}
	\left\| |\xi|^{s_2} \hat \phi_j \hat f \right\|_{L^{\gamma'}}
	\nonumber\\
	&\leq C \langle t \rangle^{-\frac{1} \sigma ( s_1 - \frac n 2) + \frac{1} \sigma ( s_2 - \frac n \gamma)}
	\| D^{s_2} \phi_j \ast f \|_{L^{\gamma}}.
	\label{eq:3.11}
	\end{align}
Similarly, the following estimates hold:
	\begin{align}
	\left\| 2^{- \frac{2n}{\sigma}} D^{s_1} \psi\left(2^{- \frac 2 \sigma} \cdot \right) \ast \mathcal S_\sigma(t) f \right\|_{L^2}
	&= \left\| e^{-\frac t 2} L_\sigma(t) |\xi|^{s_1-s_2} |\xi|^{s_2} \hat \psi(2^{\frac 2 \sigma} \cdot ) \hat f \right\|_{L^2}
	\nonumber\\
	&\leq \| e^{-\frac t 2} L_\sigma(t) |\xi|^{s_1-s_2}
	\|_{L^2\left(|\xi| < 2^{-\frac 2 \sigma}\right)}
	\left\| |\xi|^{s_2} \hat {f} \right\|_{L^\infty}
	\nonumber\\
	&\leq C \langle t \rangle^{-\frac{1} \sigma ( s_1 - \frac n 2) + \frac{1} \sigma ( s_2 - n)}
	\| D^{s_2} f \|_{L^1}.
	\label{eq:3.12}
	\end{align}
Besides
\eqref{eq:3.10} implies that for $j\in \mathbb{Z}$ with $j \geq - \frac 2 \sigma - 1$,
	\begin{align}
	\| D^{s_1} \phi_j \ast \mathcal S_\sigma(t) f \|_{L^2}
	&= e^{-\frac t 2} \left\| \langle \xi \rangle^{\frac \sigma 2} L_\sigma(t)
	\langle \xi \rangle^{-\frac \sigma 2} |\xi|^{s_1} \hat \phi_j \hat f \right\|_{L^2}
	\nonumber\\
	&\leq C e^{-\frac t 2}
	\| (1-\Delta)^{- \frac \sigma 4} D^{s_1} \phi_j \ast f \|_{L^{2}}.
	\label{eq:3.13}
	\end{align}
Therefore
\eqref{eq:2.1} follows from \eqref{eq:3.11}, \eqref{eq:3.12}, and \eqref{eq:3.13}.
We next show the estimate \eqref{eq:2.3}.
We care only the singularity of $L_\sigma$ at the origin
because the other parts may be handled as well as the proof of \cite[Lemma A.2]{IIW}.
By Lemma \ref{Lemma:3.2},
if $\beta < \frac n{q'} + \sigma$,
then the estimate
	\[
	\left\| |\cdot|^\beta \mathcal F^{-1} (e^{-t/2} L_\sigma(t,\cdot) \hat \psi) \right\|_{L^q}
	\leq C t^{\frac{-n/q'+\beta}{\sigma}}
	\]
holds.
Since we shall choose $q=\frac{2 \gamma}{3 \gamma -2}$,
the estimate above implies \eqref{eq:2.3}.
\end{proof}
Next we give a proof of Theorem \ref{Theorem:2.3}.
\begin{proof}[Proof of Theorem \ref{Theorem:2.3}]
This theorem is proved in the almost similar manner as the proof of Theorem \ref{Theorem:2.1}. We only comment the different parts of the proof.
For the operator $\widetilde{\mathcal{S}_{\sigma}}(t)$,
we have
$\sup_{t > 0} \| \partial_t L_\sigma(t,\xi)
	\|_{L^\infty\left(|\xi| > 2^{-\frac 2 \sigma}\right)}
	< \infty$
instead of \eqref{eq:3.10}.
By using this estimate and the same argument as above, we can prove the conclusion of this Theorem.
\end{proof}
Now we give a proof of Corollary \ref{Corollary:2.4}.
\begin{proof}[Proof of Corollary \ref{Corollary:2.4}]
When $p = \infty$,
by Theorems \ref{Theorem:2.1} and \ref{Theorem:2.3} and the interpolation estimate,
we have
	\begin{align}
	\| \mathcal{S}_{\sigma}(t) g \|_{L^\infty}
	&\leq C \| \mathcal{S}_{\sigma}(t) g \|_{\dot H^s}^{\frac{n}{2s}}
	\| \mathcal{S}_{\sigma}(t) g \|_{L^2}^{\frac{2s-n}{2s}}
	\nonumber\\
	&\leq C t^{-\frac n \sigma}
	( \| g \|_{L^1} + \| g \|_{\dot H^{s-\frac \sigma 2}}),
	\label{eq:3.14}\\
	\left\| \widetilde{\mathcal{S}}_{\sigma}(t) f \right\|_{L^\infty}
	&\leq C t^{-\frac n \sigma}
	( \| f \|_{L^1} + \| f \|_{\dot H^s}).
	\label{eq:3.15}
	\end{align}
When $p =1$,
as claimed in the proof of Lemma 2.3 in \cite{IIW}, the estimate
	\[
	\| f \|_{L^1}
	\leq C \| f \|_{L^2}^{1-\frac{n}{2 \beta}} \| |\cdot|^\beta f \|_{L^2}^{\frac{n}{2 \beta}}
	\]
holds provided that the right hand side is finite. This and Theorems \ref{Theorem:2.1} and \ref{Theorem:2.3} imply that
the estimates
	\begin{align}
	\| \mathcal{S}_{\sigma}(t) g \|_{L^1}
	& \leq \| \mathcal{S}_{\sigma}(t) g \|_{L^2}
	+ \left\| | \cdot |^{\beta} \mathcal{S}_{\sigma}(t) g \right\|_{L^2}
	\nonumber\\
	& \leq \| g \|_{L^1} + \| \langle \cdot \rangle^{\beta} g \|_{L^2},
	\label{eq:3.16}\\
	\left\| \widetilde{\mathcal{S}}_{\sigma}(t) f \right\|_{L^1}
	& \leq \| f \|_{L^1} + \| \langle \cdot \rangle^{\beta} f \|_{L^2}.
	\label{eq:3.17}
	\end{align}
hold.
Since $\langle \cdot \rangle^{-\beta} L^2 \hookrightarrow L^1$,
Corollary \ref{Corollary:2.4} follows from
the interpolation between \eqref{eq:3.14} and \eqref{eq:3.16},
and the interpolation between \eqref{eq:3.15} and \eqref{eq:3.17}.
\end{proof}

\section{Fundamental estimates on Besov space}
In this section, we collect two fundamental estimates on Besov spaces.

First we recall the Gagliardo-Nirenberg inequality on Besov spaces.
\begin{Lemma}[{\cite[Theorem 2.1]{HMOW}}]
\label{Lemma:4.1}
Let $n\in \mathbb{N}$, $s_0,s_1,s_2 \in \mathbb R$, $0 < p_0, p_1, p_2, q_0, q_1, q_2 \leq \infty$,
and $0 \leq \theta \leq 1$.
Then there exists a positive constant $C$ such that the Gagliardo-Nirenberg inequality
	\[
	\| f \|_{\dot B_{p_0,q_0}^{s_0}}
	\leq C \| f \|_{\dot B_{p_1,q_1}^{s_1}}^\theta
	\| f \|_{\dot B_{p_2,q_2}^{s_2}}^{1-\theta}
	\]
holds for all $f \in \dot B_{p_1,q_1}^{s_1} \cap \dot B_{p_2,q_2}^{s_2}$
if and only if
	\begin{align*}
	s_0 - \frac n p_0
	&= \theta \bigg( s_1 - \frac n p_1 \bigg)
	+ (1-\theta) \bigg( s_2 - \frac n p_2 \bigg),
	\nonumber\\
	s_0
	&\leq \theta s_1 + (1-\theta) s_2,
	\end{align*}
and
	\begin{align*}
	\frac 1 q_0
	&\leq \frac \theta {q_1} + \frac {1-\theta}{q_2}
	&&\mathrm{if} \quad p_1 \neq p_2 \ \mathrm{and} \ s_0 = \theta s_1 + (1-\theta)s_2,\\
	s_1 &\neq s_2 \ \mathrm{or} \ \frac 1 {q_0} \leq \frac \theta {q_1} + \frac {1-\theta}{q_2}
	&&\mathrm{if} \quad p_1 = p_2 \ \mathrm{and} \ s_0 = \theta s_1 + (1-\theta)s_2,\\
	s_0 - \frac n {p_0} &\neq s_1 - \frac n {p_1} \ \mathrm{or}
	\ \frac 1 {q_0} \leq \frac \theta {q_1} + \frac {1-\theta}{q_2}
	&&\mathrm{if} \quad s_0 < \theta s_1 + (1-\theta)s_2.
	\end{align*}
\end{Lemma}

Next,
we recall the fractional Leibniz rule on Besov spaces:
\begin{Lemma}
\label{Lemma:4.2}
Let $ 1 \leq  p_0,p_1,p_2,p_3,p_4 \leq \infty$ and $s \geq 0$.
For any $f \in \dot B_{p_1,2}^s \cap L^{p_2}$
and $g \in \dot B_{p_3,2}^s \cap L^{p_4}$,
the estimate
	\[
	\| fg \|_{\dot B_{p_0,2}^s}
	\leq C \| f \|_{\dot B_{p_1,2}^s} \| g \|_{L^{p_4}}
	+ C \| g \|_{\dot B_{p_3,2}^s} \| f \|_{L^{p_2}}.
	\]
holds if
	\[
	\frac 1 {p_0}
	= \frac 1 {p_1} + \frac 1 {p_4}
	= \frac 1 {p_2} + \frac 1 {p_3}.
	\]
\end{Lemma}
For the proof of this lemma, see Appendix \ref{section:B}.

\section{Estimates for the Duhamel term}
\subsection{Definitions of the solution space and its auxiliary space}
In this subsection, we introduce our solution space and its auxiliary space for the construction of solutions to the nonlinear problem \eqref{eq:1.10}.
Let $\sigma>0$ and $r \in [1,2]$. We introduce a dilation operator $\mathcal D_{\sigma,r}(t)$ given by
	\[
	\mathcal D_{\sigma,r} (t) u(t,x)
	:= \langle t \rangle^{ \frac n{r\sigma}} u \left(t ,\langle t \rangle^{\frac 1 \sigma} x\right).
	\]
Let 
\begin{equation}
\label{eq:5.1}
    q_\sigma := \max\left(1, \frac{2 n}{n+\sigma}\right) = \frac{n + \max(n,\sigma)}{n+\sigma},
\end{equation}
$s\ge 0$, $\alpha > 0$ and $\gamma\in [1,\infty]$. Then for $T\in (0,\infty]$, we define the norms
	\begin{align*}
	\| f \|_{X_0^{s,\alpha}}
	:=& \| f \|_{L^2} + \| f \|_{\dot H^s} + \| | \cdot |^\alpha f \|_{L^2},\\
	\| f \|_{Y_0^{s,\alpha,\gamma}}
	:=&
	\| | \cdot |^\alpha f \|_{L^{q_\sigma}} + \| f \|_{L^{\gamma}}
		+ \begin{cases}
		\| f \|_{\dot B_{q_\sigma,2}^s}
		&\mathrm{if} \quad s > 0,\\
		\| f \|_{L^{q_\sigma}}
		&\mathrm{if} \quad s = 0,\\
		\end{cases}\\
	\| u \|_{X^{s,\alpha,r}(T)}
	:=& \sup_{0 < t < T} \left\| \mathcal D_{\sigma,r}(t) u(t) \right\|_{X_0^{s,\alpha}},\\
	\| v \|_{Y^{s,\alpha,\gamma,r}(T)}
	:=&
	\sup_{0 < t < T} \langle t \rangle^{\frac {n \rho}{\sigma r}}
	\left\| \langle t \rangle^{- \frac {n}{\sigma r}} \mathcal D_{\sigma,r}(t) v(t) \right\|_{Y_0^{s,\alpha,\gamma}}.
	\end{align*}
We remark that $Y_0^{s,\alpha,\gamma}$ consists of $L^{q_\sigma}$
so as to apply the embedding
	\begin{align}
	L^{q_\sigma} \hookrightarrow H^{-\frac \sigma 2}.
	\label{eq:5.2}
	\end{align}
 Here $H^{-\frac{\sigma}{2}}$ appears in the second term in the right hand side of (\ref{eq:2.1}), for more details, see Theorem \ref{Theorem:2.1} and Lemma \ref{Lemma:5.1}. We also remark that
the identity
	\[
	\| \mathcal N(u) \|_{Y^{s,\alpha,\gamma,r}(T)}
	= \sup_{0 < t < T} \langle t \rangle^{\frac {n\rho}{\sigma r}}
	\| \mathcal N( \langle t \rangle^{- \frac {n}{\sigma r}} \mathcal D_{\sigma,r} (t) u)
	\|_{Y_0^{s,\alpha,\gamma}}
	\]
holds,
so if the identity $|\mathcal N(C z)| = C^{\rho} |\mathcal N(z)|$ holds for any $z\in \mathbb{R}$, then we have
	\[
	\| \mathcal N(u) \|_{Y^{s,\alpha,\gamma,r}(T)}
	= \sup_{0 < t < T} \| \mathcal N( \mathcal D_{\sigma,r} (t) u) \|_{Y_0^{s,\alpha,\gamma}}.
	\]

As we see Corollary \ref{Corollary:6.6},
thanks to the dilation operator $\mathcal D_{\sigma,r}(t)$,
it is sufficient to control the $Y_0^{s,\alpha,\gamma}$ norm of the nonlinearity $\mathcal N^{(\overline s)}$
by using $X_0^{s,\alpha}$ norm of solutions for some fixed $t$
in order to control the $Y^{s,\alpha,\gamma,r}(T)$ norm of the nonlinearity
by using $X^{s,\alpha,\gamma}(T)$ of solutions for any $0 < t < T$.

\subsection{Estimates for the Duhamel term}

Next, we show the following estimates for the Duhamel term
by applying Theorem \ref{Theorem:2.1}.
\begin{Lemma}[Estimate of the Duhamel term]
\label{Lemma:5.1}
Let $n\in \mathbb{N}$, $\sigma>0$, $\alpha>0$ and $r\in [1,2]$. We assume that the estimates
	\begin{align}
	\frac \alpha n + \frac 1 2
	> \frac 1 r
	> \begin{cases}
	0
	&\mathrm{if} \quad \sigma = 2 \mathbb N,\\
	\frac{2 \alpha + n - 4 \sigma}{2n}
	&\mathrm{otherwise}
	\end{cases}
	\label{eq:5.3}
	\end{align}
hold. Let $\gamma\in \left[1,\min(r,q_\sigma)\right]$, where $q_{\sigma}$ is given by (\ref{eq:5.1}), satisfy the estimates
	\begin{align}
	\frac{rn}{n+r\sigma}
	\leq \gamma
	<
	\begin{cases}
	\infty
	&\mathrm{if} \quad \sigma = 2 \mathbb N,\\
	\frac{2n}{( 2\alpha + n - 2\sigma)_+}
	&\mathrm{otherwise}.
	\end{cases}
	\label{eq:5.4}
	\end{align}
Then there exists a positive constant $C>0$ such that for $T \in (0,1)$, the estimate
	\begin{align}
	\bigg\| \int_0^t \mathcal S_\sigma(t-\tau) v(\tau) d \tau \bigg\|_{X^{s,\alpha,r}(T)}
	&\leq C \int_0^T \mathcal \| v \|_{Y^{s,\alpha,\gamma,r}(\tau)} d \tau
	\label{eq:5.5}
	\end{align}
holds.
For any $T > 0$, the following estimate also holds:
	\begin{align}
	&\bigg\| \int_0^t \mathcal S_\sigma(t-\tau) v(\tau) d \tau \bigg\|_{X^{s,\alpha,r}(T)}
	\nonumber\\
	&\leq C \| v \|_{Y^{s,\alpha,r,\gamma}(T)}
	\begin{cases}
	1 + \log\langle T \rangle
	& \mathrm{if} \quad (\gamma,\rho) = (r, 1+\frac{ r \sigma}{n}),\\
	\langle T \rangle^{\left(1- \frac {n(\rho-1)}{r\sigma}\right)_+}
	& \mathrm{otherwise}.
	\end{cases}
	\label{eq:5.6}
	\end{align}
\end{Lemma}

\begin{proof}
We note that the second estimate of \eqref{eq:5.3} guarantees
the existence of $\gamma$ satisfying \eqref{eq:5.4}.
For simplicity,
we write $\mu(q_1,q_2) := \frac n \sigma( \frac 1 {q_2} - \frac 1 {q_1})$
for any $1 \leq q_1, q_2 \leq \infty$.
Then $X^{s,\alpha,r}(T)$ and $Y^{s,\alpha,\gamma,r}(T)$ norms are rewritten by
\begin{align*}
	\| u \|_{X^{s,\alpha,r}(T)}
	=& \sup_{0<t<T} \langle t \rangle^{\mu(2,r)}
	\Big\lbrack
	\| u(t) \|_{L^2}
	+ \langle t \rangle^{\frac s \sigma} \| u(t) \|_{\dot H^s}
	+ \langle t \rangle^{- \frac \alpha \sigma}
		\| |\cdot|^\alpha u(t) \|_{L^2}
	\Big\rbrack,\\
	\| v \|_{Y^{s,\alpha,\gamma,r}(T)}
	=&
		\sup_{0<t<T}
		\left( \langle t \rangle^{\mu(q_\sigma, \frac r p) - \frac \alpha \sigma } \| |\cdot|^\alpha v(t) \|_{L^{q_\sigma}}
		+ \langle t \rangle^{\mu(\gamma,\frac r p)} \| v(t) \|_{L^{\gamma}}\right)\\
		&+\begin{cases}
		\displaystyle
		\sup_{0<t<T}
		\langle t \rangle^{\mu(q_\sigma, \frac r \rho) + \frac s \sigma} \| v(t) \|_{\dot B_{q_\sigma,2}^s}
		&\mathrm{if} \quad s > 0,\\
		\displaystyle
		\sup_{0<t<T}
		\langle t \rangle^{\mu(q_\sigma, \frac r \rho)} \| v(t) \|_{L^{q_\sigma}}
		&\mathrm{if} \quad s = 0.
	\end{cases}
	\end{align*}
Here we recall the estimates
	\[
	q_\sigma := \frac{n+\max(n,\sigma)}{n+\sigma}
	\geq \frac{2n}{n+\sigma}.
	\]
We note that since $\gamma \leq q_\sigma$,
for any $s > 0$, there exists a positive constant $C=C(s)>0$ such that
the estimate
	\begin{align}
	\| f \|_{L^{q_\sigma}}
	\leq C \| f \|_{L^{\gamma}}^{\theta} \| f \|_{\dot B_{q_\sigma,2}^s}^{1-\theta}
	\label{eq:5.7}
	\end{align}
holds for any $f\in \dot B_{q_\sigma,2}^s(\mathbb{R}^n)\cap L^{\gamma}(\mathbb{R}^n)$ for some $\theta\in [0,1]$. This follows from Lemma \ref{Lemma:4.1}
with
$\left(\dot B_{p_0,q_0}^{s_0}, \dot B_{p_1,q_1}^{s_1}, \dot B_{p_2,q_2}^{s_2}\right)
=\left(\dot B_{q_\sigma,1}^0, \dot B_{q_\sigma,2}^s, \dot B_{\gamma,\infty}^0\right)$
and $\theta \in [0,1] $ satisfying
	\[
	- \frac{n}{q_\sigma}
	= \theta \left(- \frac{n}{\gamma}\right) + (1-\theta) \left( s - \frac{n}{q_\sigma}\right).
	\]
Here we have used the fact that the relations $\dot B_{q_\sigma,1}^0 \hookrightarrow L^{q_\sigma}$
and $L^{\gamma} \hookrightarrow \dot B_{\gamma,\infty}^0$
hold (See \cite[6.3.1. Theorem]{INT}, for example).

Now we divide the integral region $[0,t]$ into the two parts $[0,t/2]$ and $[t/2,t]$ and we write
	\[
	I_1 := \int_{t/2}^t \mathcal S_\sigma (t-\tau) v(\tau) d \tau
	\quad \mathrm{and} \quad
	I_2 := \int_{0}^{t/2} \mathcal S_\sigma (t-\tau) v(\tau) d \tau.
	\]

We estimate $I_1$.
Theorem \ref{Theorem:2.1} with $\gamma=\nu=q_\sigma$, \eqref{eq:5.2}, and \eqref{eq:5.7}
imply that the following estimates
	\begin{align}
	\| \mathcal S_\sigma(t-\tau) v(\tau) \|_{L^2}
	&\leq C \langle t-\tau \rangle^{-\mu(2,q_\sigma)} \| v(\tau) \|_{L^{q_\sigma}}
	\nonumber\\
	&+ e^{-\frac{t-\tau}{4}}
	\| \mathcal F^{-1} ( ( 1 - \hat \psi(2^{- [ \frac 2 \sigma + 1]} \cdot )) \hat v(\tau)) \|_{H^{- \frac \sigma 2}}
	\nonumber\\
	&\leq C \| v \|_{Y^{s,\alpha,\gamma,r}(\tau)}
		\langle t-\tau \rangle^{-\mu(2,q_\sigma)}
		\langle \tau \rangle^{-\mu(q_\sigma,\frac r p)}.
	\label{eq:5.8}
	\end{align}
for any $\tau\in [0,t]$. Indeed,
\eqref{eq:5.8} follows from the estimate
	\[
	\| v(\tau) \|_{L^{\gamma}}^{\theta} \| v(\tau) \|_{\dot B_{q_\sigma,2}^s}^{1-\theta}
	\leq C \langle \tau \rangle^{-\mu(q_\sigma,\frac r p)} \| v \|_{Y^{s,\alpha,\gamma,r}(\tau)}
	\]
for $s > 0$.
Theorem \ref{Theorem:2.1} also implies that
for any $s>0$ and $\alpha > 0$, the estimates
	\begin{align*}
	\| D^s \mathcal S_\sigma(t-\tau) v(\tau) \|_{L^2}
	&\leq C \langle t-\tau \rangle^{-\mu(2,q_\sigma)}
	\| v(\tau) \|_{\dot B_{q_\sigma,2}^s}\\
	&\leq C \| v \|_{Y^{s,\alpha,\gamma,r}(\tau)}
		\langle t-\tau \rangle^{-\mu(2,q_\sigma)}
		\langle \tau \rangle^{-\mu(q_\sigma,\frac r \rho) - \frac s \sigma},\\
	\| |\cdot|^\alpha \mathcal S_\sigma(t-\tau) v(\tau) \|_{L^2}
	&\leq C \langle t-\tau \rangle^{\frac \alpha \sigma -\mu(2,\gamma)} \| v(\tau) \|_{L^{\gamma}}\\
	&+ C \langle t-\tau \rangle^{-\mu(2,q_\sigma)} \| | \cdot |^\alpha v(\tau) \|_{L^{q_\sigma}}
	+ C e^{-\frac{t-\tau}{4}}\| v(\tau) \|_{L^{q_\sigma}}\\
	&\leq C \| v \|_{Y^{s,\alpha,\gamma,r}(\tau)}
	\langle t-\tau \rangle^{\frac \alpha \sigma -\mu(2,\gamma)}
		\langle \tau \rangle^{-\mu(\gamma,\frac r \rho)}\\
	&+ C \| v \|_{Y^{s,\alpha,\gamma,r}(\tau)}
		\langle t-\tau \rangle^{-\mu(2,q_\sigma)}
		\langle \tau \rangle^{-\mu(q_\sigma,\frac r \rho) + \frac \alpha \sigma}.
	\end{align*}
hold for any $\tau\in [0,t]$. We note that the second estimate of \eqref{eq:5.4} implies that the condition \eqref{eq:2.2} with $\beta=\alpha$ is satisfied.
Therefore, for any $0 < t < 1$, the estimate
	\begin{align}
	&\langle t \rangle^{\mu(2,r)}
	( \| I_1 \|_{L^2}
	+ \langle t \rangle^{\frac s \sigma} \| D^s I_1 \|_{L^2}
	+ \langle t \rangle^{-\frac \alpha \sigma} \| |\cdot|^\alpha I_1 \|_{L^2})
	\nonumber\\
	&\leq C \int_0^t \| v\|_{Y^{s,\alpha,\gamma,r}(\tau)} d \tau
	\label{eq:5.9}
	\end{align}
holds. Moreover, for any $t \in [0,T]$, the estimates
	\begin{align}
	&\langle t \rangle^{\mu(2,r)}
	( \| I_1 \|_{L^2}
	+ \langle t \rangle^{\frac s \sigma} \| D^s I_1 \|_{L^2}
	+ \langle t \rangle^{-\frac \alpha \sigma} \| |\cdot|^\alpha I_1 \|_{L^2})
	\nonumber\\
	&\leq C \| v \|_{Y^{s,\alpha,\gamma,r}(T)}
		\langle t \rangle^{\mu(2,r)-\mu(q_\sigma,\frac r p)}
		\int_{\frac t 2}^t \langle t - \tau \rangle^{- \mu(2,q_\sigma)} d \tau
	\nonumber\\
	&+ C \| v \|_{Y^{s,\alpha,\gamma,r}(T)}
		\langle t \rangle^{\mu(2,r)-\mu(\gamma,\frac r \rho) - \frac \alpha \sigma}
		\int_{\frac t 2}^t \langle t - \tau \rangle^{\frac \alpha \sigma - \mu(2,\gamma)} d \tau
	\nonumber\\
	&\leq C \| v \|_{Y^{s,\alpha,\gamma,r}(T)}
		\langle t \rangle^{1-\frac{n(\rho-1)}{\sigma r}},
	\label{eq:5.10}
	\end{align}
hold, where we have used the identities
	\begin{align*}
	\mu(2,r) - \mu\left(q_\sigma, \frac r \rho\right) - \mu(2,q_\sigma)
	&= - \mu(r,2) - \mu\left(2, \frac r \rho\right)
	= - \mu\left(r,\frac r \rho\right)
	= - \frac{n(\rho-1)}{\sigma r},\\
	\mu(2,r) - \mu\left(\gamma, \frac r \rho\right) - \mu(2,\gamma)
	&= - \frac{n(p-1)}{\sigma r},
	\end{align*}
and the estimates
	\begin{align*}
	\mu(2, q_\sigma)
	&\leq \mu\left(2, \frac{2n}{n+\sigma}\right)
	= \frac n \sigma \bigg( \frac{n+\sigma}{2n} - \frac 1 2 \bigg)
	= \frac 1 2,\\
	\frac \alpha \sigma - \mu(2,\gamma)
	&> \frac n {\sigma r}  - \frac n {2 \sigma} + \frac n {2 \sigma} - \frac n {\sigma \gamma}
	\geq \frac n {\sigma } \left( \frac 1 r - \frac 1 \gamma\right)
	\geq -1,
	\end{align*}
which follows from the inequality $q_\sigma > \frac{2n}{n+\sigma}$ and the first estimates of \eqref{eq:5.3} and \eqref{eq:5.4}.

Next, we estimate $I_2$.
Theorem \ref{Theorem:2.1}, \eqref{eq:5.2}, \eqref{eq:5.7}, and $\gamma \leq q_\sigma$
imply that there exists a positive constant $C>0$ such that the estimates
	\begin{align*}
	\| \mathcal S_\sigma(t-\tau) v(\tau) \|_{L^2}
		&\leq C ( \langle t-\tau \rangle^{-\mu(2,\gamma)} \| v(\tau) \|_{L^{\gamma}}
		+ e^{-\frac{t-\tau}{4}} \| v(\tau) \|_{L^{q_\sigma}})\\
	&\leq C \| v \|_{Y^{s,\alpha,\gamma,r}(T)}
		(\langle t-\tau \rangle^{-\mu(2,\gamma)}
		\langle \tau \rangle^{-\mu(\gamma,\frac r \rho)}
		+ e^{-\frac{t-\tau}{4}} \langle \tau \rangle^{-\mu(q_\sigma,\frac r \rho)})\\
	&\leq C \| v \|_{Y^{s,\alpha,\gamma,r}(T)}
		\langle t-\tau \rangle^{-\mu(2,\gamma)}
		\langle \tau \rangle^{-\mu(\gamma,\frac r \rho)}
	\end{align*}
hold for any $\tau\in [0,t]$. Similarly, for any $s > 0$, there exists a positive constant $C>0$ such that the inequalities
	\begin{align*}
	&\| D^s \mathcal S_\sigma(t-\tau) v(\tau) \|_{L^2}\\
	&\leq C (
		\langle t-\tau \rangle^{-\mu(2,\gamma) - \frac s \sigma} \|v(\tau) \|_{L^{\gamma}}
		+ e^{-\frac{t-\tau}{4}}
		\| D^s \mathcal F^{-1} ( 1 - \hat \psi(2^{- [ \frac 2 \sigma + 1]} \cdot ) \hat v(\tau)) \|_{H^{- \frac \sigma 2}}
		)\\
	&\leq C ( \langle t-\tau \rangle^{-\mu(2,\gamma) - \frac s \sigma}
		\|v(\tau) \|_{L^{\gamma}}
		+ e^{-\frac{t-\tau}{4}}
		\| v(\tau) \|_{\dot B_{q_\sigma 2}^s})\\
	&\leq C \| v \|_{Y^{s,\alpha,\gamma,r}(\tau)}
		\langle t-\tau \rangle^{-\mu(2,\gamma) - \frac s \sigma}
		\langle \tau \rangle^{-\mu(\gamma,\frac r \rho)}
	\end{align*}
hold, and for any $\alpha > 0$, the estimates
	\begin{align*}
	&\| |\cdot|^\alpha \mathcal S_\sigma(t-\tau) v(\tau) \|_{L^2}\\
	&\leq C (
		\langle t-\tau \rangle^{-\mu(2,\gamma) + \frac \alpha \sigma} \| v(\tau) \|_{L^{\gamma}}
		+ \langle t-\tau \rangle^{-\mu(2,q_\sigma)} \| |\cdot|^\alpha v(\tau) \|_{L^{q_\sigma}}
		+ e^{-\frac{t-\tau}{4}} \| v(\tau) \|_{L^{q_\sigma}}
	)\\
	&\leq C \| v \|_{Y^{s,\alpha,\gamma,r}(\tau)}
		\left(\langle t-\tau \rangle^{-\mu(2,\gamma) + \frac \alpha \sigma}
		\langle \tau \rangle^{-\mu(\gamma,\frac r p)}
		+ \langle t-\tau \rangle^{-\mu(2,q_\sigma)}
		\langle \tau \rangle^{-\mu(q_\sigma,\frac r \rho) + \frac \alpha \sigma}\right)
	\end{align*}
hold. Therefore,
for any $0 < t < 1$, the estimate
	\begin{align}
	&\langle t \rangle^{\mu(2,r)}
	( \| I_2 \|_{L^2}
	+ \langle t \rangle^{\frac s \sigma} \| D^s I_2 \|_{L^2}
	+ \langle t \rangle^{-\frac \alpha \sigma} \| |\cdot|^\alpha I_2 \|_{L^2})
	\nonumber\\
	&\leq C \int_0^t \| v \|_{Y^{s,\alpha,\gamma,r}(\tau)} d \tau
	\label{eq:5.11}
	\end{align}
holds.
Besides for any $ t \in [0,T]$,
we estimate $I_2$. By the inequality $r \geq \gamma$, the estimates
	\begin{align}
	&\langle t \rangle^{\mu(2,r)}
	( \| I_2 \|_{L^2}
	+ \langle t \rangle^{\frac s \sigma} \| D^s I_2 \|_{L^2}
	+ \langle t \rangle^{-\frac \alpha \sigma} \| |\cdot|^\alpha I_2 \|_{L^2})
	\nonumber\\
	&\leq C \| v \|_{Y^{s,\alpha,\gamma,r}(T)}
	\nonumber\\
	&\cdot \bigg(
			\langle t \rangle^{- \mu(r,\gamma)}
				\int_0^{\frac t 2} \langle \tau \rangle^{- \mu(\gamma, \frac r \rho)} d \tau
			+ \langle t \rangle^{- \mu(r,q_\sigma) - \frac \alpha \sigma}
				\int_0^{\frac t 2} \langle \tau \rangle^{- \mu(q_\sigma, \frac r \rho) + \frac \alpha \sigma} d \tau
		\bigg)
	\label{eq:5.12}
	\end{align}
hold. Here we note that the estimate $\mu(r,\gamma) > 0$ holds if $r > \gamma$
and the first estimate of \eqref{eq:5.3} implies that the inequalities
	\[
	\mu(r,q_\sigma) + \frac \alpha \sigma
	= \frac n \sigma \left( \frac{1}{q_\sigma} - \frac 1 r + \frac \alpha n\right)
	> 0
	\]
hold. We claim that for any $a \in \mathbb R$ and $d > 0$, there exists a positive constant $C=C(a,d)>0$ such that the estimate
	\begin{align}
	\langle t \rangle^{-d} \int_0^t \langle \tau \rangle^a d \tau
	\leq C \langle t \rangle^{(1+a-d)_+}
	\label{eq:5.13}
	\end{align}
holds.
When $a = -1$, \eqref{eq:5.13} holds obviously.
When $a \neq -1$, \eqref{eq:5.13} follows from the fact that
	\[
	(1+a)_+ - d \leq (1+a-d)_+.
	\]
Since the equalities
	\begin{align*}
	1 - \mu(r,\gamma) - \mu\left(\gamma,\frac r \rho\right)
	&= 1 - \mu\left(r, \frac r \rho\right) = 1 - \frac{n(\rho-1)}{r\sigma},\\
	1 - \mu(r,q_\sigma) + \frac \alpha \sigma - \mu\left(q_\sigma,\frac r \rho\right) - \frac \alpha \sigma
	&= 1 - \mu\left(r, \frac r \rho\right) = 1 - \frac{n(\rho-1)}{r\sigma}
	\end{align*}
hold,
the estimate
	\begin{align}
	&\langle t \rangle^{\mu(2,r)}
	( \| I_2 \|_{L^2}
	+ \langle t \rangle^{\frac s \sigma} \| D^s I_2 \|_{L^2}
	+ \langle t \rangle^{-\frac \alpha \sigma} \| |\cdot|^\alpha I_2 \|_{L^2})
	\nonumber\\
	&\leq C \| v(t) \|_{Y^{s,\alpha,\gamma,r}}
	\langle t \rangle^{\left(1-\frac{n(\rho-1)}{r\sigma}\right)_+}
	\label{eq:5.14}
	\end{align}
follows from \eqref{eq:5.12} and \eqref{eq:5.13} if $r > \gamma$.
When $r=\gamma$,
\eqref{eq:5.12} and \eqref{eq:5.13} imply that the estimate
	\begin{align}
	&\langle t \rangle^{\mu(2,r)}
	( \| I_2 \|_{L^2}
	+ \langle t \rangle^{\frac s \sigma} \| D^s I_2 \|_{L^2}
	+ \langle t \rangle^{-\frac \alpha \sigma} \| |\cdot|^\alpha I_2 \|_{L^2})
	\nonumber\\
	&\leq C \| v \|_{Y^{s,\alpha,\gamma,r}(T)}
	\bigg( 1 + \int_0^{\frac t 2} \langle \tau \rangle^{- \mu(r, \frac r \rho)} d \tau \bigg).
	\label{eq:5.15}
	\end{align}
holds. 
Therefore, \eqref{eq:5.5} follows from \eqref{eq:5.9} and \eqref{eq:5.11}.
Since $\mu(r, \frac r p) = 1$ holds if and only if $\rho = 1 + \frac{r\sigma}{n}$,
\eqref{eq:5.6} follows from \eqref{eq:5.10}, \eqref{eq:5.14}, and \eqref{eq:5.15}, which completes the proof of the lemma.
\end{proof}

\section{Estimates for nonlinearity}
In this subsection,
we estimate the auxiliary norm $Y^{s,\alpha,\gamma,r}(T)$ for the nonlinearity.

In Lemma \ref{Lemma:6.1} bellow,
we show that some weighted Lebesgue spaces and a homogeneous Sobolev space
are embedded into the solution space $X^{s,\alpha}_0$ for $s,\alpha\ge 0$:

\begin{Lemma}
\label{Lemma:6.1}
Let $n\in \mathbb{N}$, $s \geq 0$, $\alpha \geq 0$. Then the embedding holds:
	\begin{align}
	X_0^{s,\alpha} &\hookrightarrow \dot H^{s'},
	\quad \forall s' \in \lbrack 0,s \rbrack,
	\label{eq:6.1}\\
	X_0^{s,\alpha} &\hookrightarrow |\cdot|^{-\alpha'} L^{2},
	\quad \forall \alpha' \in \lbrack 0,\alpha \rbrack,
	\label{eq:6.2}\\
	X_0^{s,\alpha}
	&\hookrightarrow L^\nu,
	\label{eq:6.3}
	\end{align}
where
	\begin{align}
	\frac{2n}{2\alpha+n}
	< \nu
	\begin{cases}
	\leq \infty
	& \mathrm{if} \quad s > n/2,\\
	< \infty
	& \mathrm{if} \quad s = n/2,\\
	\leq \frac{2n}{n-2s}
	& \mathrm{if} \quad s < n/2.
	\end{cases}
	\label{eq:6.4}
	\end{align}
\end{Lemma}

\begin{proof}
By the definition of the solution space $X_0^{s,\alpha}$, the embedding $X_0^{s,\alpha} \hookrightarrow L^2$ holds for any $s \geq 0$ and $\alpha \geq 0$.

For any $s'\in [0,s]$, the embedding \eqref{eq:6.1} holds, since by the interpolation estimate, there exists a positive constant $C=C(s,s')$ such that the inequalities
	\[
	\| D^{s'} f \|_{L^2}
	\leq C \| D^s f \|_{L^2}^{\frac{s'}{s}} \| f \|_{L^2}^{1-\frac{s'}{s}}
	\leq C \| f \|_{X_0^{s,\alpha}}
	\]
hold for any $ f \in X_0^{s,\alpha}$.

For any $\alpha'\in [0,\alpha]$, the embedding \eqref{eq:6.2} holds, because by the H\"older inequality, the estimates
	\[
	\| |\cdot|^{\alpha'} f \|_{L^2}
	\leq
	\| f \|_{L^2}^{1-\frac{\alpha'}\alpha}
	\| | \cdot|^\alpha f \|_{L^2}^{\frac{\alpha'}\alpha}
	\leq
	\| f \|_{X_0^{s,\alpha}}
	\]
holds for any $ f \in X_0^{s,\alpha}$.

When $\nu \geq 2$, the embedding \eqref{eq:6.3} follows from the Gagliardo-Nirenberg inequality and the assumption of $\nu$. When $\nu \in (\frac{2n}{2\alpha+n} ,2 )$, by the H\"older inequality, the estimate
	\[
	\| f \|_{L^\nu}
	\leq
	\| \langle \cdot \rangle^{-\alpha} \|_{L^{\frac{2 \nu}{2-\nu}}}
	\| \langle \cdot \rangle^{\alpha} f \|_{L^2}
	\]
holds for any $f\in X_0^{s,\alpha}$and $\langle \cdot \rangle^{-\alpha} \in L^{\frac{2 \nu}{2-\nu}}$
if $\nu$ satisfies $\nu>\frac{2n}{2\alpha+n}$, which implies \eqref{eq:6.3}. This completes the proof of the lemma.
\end{proof}

\begin{Lemma}
\label{Lemma:6.2}
Let $n\in \mathbb{N}$, $\alpha, s \geq 0$ and $\rho>1$.
We assume that $\rho$ satisfies \eqref{eq:2.5} and the inequality
\begin{align}
	\rho \leq 1 + \frac{\min(n,\sigma)}{(n-2s)_+}
	\label{eq:6.5}
	\end{align}
and the function $\mathcal{N}$ satisfies the estimate (\ref{eq:2.8}) with $\bar{s}=0$. Then there exists a positive constant $C=C(\alpha,s,\rho)>0$ such that the estimates
	\begin{align}
	\| \mathcal N(f) - \mathcal N(g) \|_{L^{q_\sigma}}
	&\leq C (\| f \|_{X_0^{s,\alpha}} + \| g \|_{X_0^{s,\alpha}})^{\rho-1} \| f - g \|_{L^2},
	\label{eq:6.6}\\
	\| | \cdot |^\alpha ( \mathcal N(f) - \mathcal N(g) ) \|_{L^{q_\sigma}}
	&\leq C (\| f \|_{X_0^{s,\alpha}} + \| g \|_{X_0^{s,\alpha}})^{\rho-1} \| | \cdot |^\alpha (f - g) \|_{L^2}
	\label{eq:6.7}
	\end{align}
hold for any $f, g \in X_0^{s,\alpha}$.
\end{Lemma}

\begin{proof}
By the assumption (\ref{eq:2.8}) and the H\"older inequality, the inequalities
	\begin{align*}
	\| \mathcal N(f) - \mathcal N(g) \|_{L^{q_\sigma}}
	&\leq C \| (|f|+|g|)^{\rho-1} (f-g) \|_{L^{q_\sigma}}\\
	&\leq C (\| f \|_{L^{\nu}} + \| g \|_{L^\nu})^{\rho-1} \| f - g \|_{L^{2}},
	\end{align*}
hold, where $\nu \in (0,\infty)$ is given by
	\[
	\frac{1}{q_\sigma}
	= (\rho-1) \frac{1}{\nu} + \frac{1}{2}.
	\]
This is rewritten by
	\begin{align}
	\nu := \frac{2n}{\min(n,\sigma)} (\rho-1).
	\label{eq:6.8}
	\end{align}
Then \eqref{eq:2.5}, \eqref{eq:6.5}, and \eqref{eq:6.8} imply that the inequalities
	\[
	\frac{2n}{2\alpha+n}
	< \nu \leq \frac{2n}{(n-2s)_+},
	\]
hold, which coincides with \eqref{eq:6.4}.
Therefore Lemma \ref{Lemma:6.1} implies \eqref{eq:6.6}.
\eqref{eq:6.7} is proved similarly, which completes the proof of the lemma.
\end{proof}

Next we show nonlinear estimates from a homogeneous Besov space $\dot B_{P,2}^s$ to the solution spaces $X_0^{s,\alpha}$ and $X_0^{s',\alpha}$, where the exponents $s,s'$ and $P$ are given in Lemma \ref{Lemma:6.3}. More precisely, in Lemma \ref{Lemma:6.3} below,
we show the estimates for a nonlinearity $\mathcal M$, whose $\mathcal M$ is substituted by $\mathcal N^{(\overline s)}$ with some integer $\overline s$
in the proof of Lemma \ref{Lemma:6.5} below.

\begin{Lemma}
\label{Lemma:6.3}
Let $n\in \mathbb{N}$, $s \geq s'> 0$, $\rho_0 > 1$ and $\alpha \geq 0$.
Let $s_0$ be a positive number satisfying
$s_0 < \min(2,\rho_0)$ and $s_0 \leq s$.
Let $\mathcal M \in C^1(\mathbb R)$ satisfy $\mathcal M^{(\ell)}(0) = 0$ for $\ell=0,1$ and for $\ell=0,1$, there exists a positive constant $C=C(\ell)>0$ such that the estimate
	\begin{align}
	|\mathcal M^{(\ell)}(x) - \mathcal M^{(\ell)}(y) |
	\leq
	\begin{cases}
	C (|x|+|y|)^{\rho_0-\ell-1} |x-y|
	&\mathrm{if} \quad \rho_0 - \ell \geq 1,\\
	C |x-y|^{\rho_0-\ell}
	&\mathrm{if} \quad \rho_0 - \ell < 1
	\end{cases}
	\label{eq:6.9}
	\end{align}
holds for any $x,y\in \mathbb{R}$.
We also assume that $P \geq 1$ satisfies
	\begin{align}
	\frac 1 P
	&<
	\frac{n+2\alpha}{2n} (\rho_0-1) + \frac 1 2 + \frac{s-s_0}{s} \frac{\min(2\alpha,n)}{2n},
	\label{eq:6.10}\\
	\frac 1 P
	&\begin{cases}
	\displaystyle
	\geq \left(\frac 1 2 - \frac{s-s_0}{n}\right)_+
	+ (\rho_0-1)\left(\frac 1 2 - \frac{s'}{n}\right)_+
	&\mathrm{if} \quad s' \neq \frac n 2,\\
	\displaystyle
	> \left(\frac 1 2 - \frac{s-s_0}{n}\right)_+
	&\mathrm{if} \quad s' = \frac n 2.
	\end{cases}
	\label{eq:6.11}
	\end{align}
Then there exists a positive constant $C_1>0$ such that the estimate
	\begin{align}
	\| \mathcal M(f) \|_{\dot B_{P,2}^{s_0}}
	\leq C_1 \| f \|_{X_0^{s,\alpha}}
	\| f \|_{X_0^{s',\alpha}}^{\rho_0-1}
	\label{eq:6.12}
	\end{align}
holds for all $f \in X_0^{s,\alpha}$.
Moreover,
for any positive number $d_{\rho_0,s_0} \leq 1$ satisfying
	\[
	d_{\rho_0,s_0}
	\begin{cases}
	\leq \rho_0- 1
	&\mathrm{if} \quad 0 < s_0 < 1,\\
	< \rho_0- s_0
	&\mathrm{if} \quad s_0 \geq 1,
	\end{cases}
	\]
there exists a positive constant $C_{d_{\rho_0,s_0}}$ such that
	\begin{align}
	\| \mathcal M(f) - \mathcal M(g) \|_{\dot B_{P,2}^{s_0}}
	&\leq C_{d_{\rho_0,s_0}}
	(\| f \|_{X_0^{s',\alpha}} + \| g \|_{X_0^{s',\alpha}}
	)^{\rho_0-1}
	\| f - g \|_{X_0^{s,\alpha}}
	\nonumber\\
	&+ C_{d_{\rho_0,s_0}}
	(\| f \|_{X_0^{s,\alpha}} + \| g \|_{X_0^{s,\alpha}}
	)^{\rho_0-d_{p_0,s_0}}
	\| f - g \|_{X_0^{s',\alpha}}^{d_{\rho_0,s_0}}
	\label{eq:6.13}
	\end{align}
holds for any $f, g \in X_0^{s,\alpha}$.
\end{Lemma}

\begin{Remark}
The above condition on
$d_{\rho_0,s_0}$
says that
$d_{\rho_0, s_0}<1$
must hold
if $\rho_0 - s_0 \le 1$.
However, when $\rho_0 - s_0 = 1$, it is naturally expected that we can take
$d_{\rho_0, s_0}=1$.
Indeed, when $\rho_0 \ge 2$,
by modifying the conditions
\eqref{eq:6.10}--\eqref{eq:6.11} to slightly stronger ones,
we can actually take
$d_{\rho_0, s_0}=1$
even if
$\rho_0 - s_0 = 1$.
However, Lemma \ref{Lemma:6.3} is sufficient to construct a solution to \eqref{eq:1.10},
and we do not discuss it here.
\end{Remark}

\begin{proof}
We set two positive numbers $P_0$ and $P_1$ satisfying
	\[
	\frac 1 P = \frac 1 {P_0} + \frac 1 {P_1}.
	\]
The assumption \eqref{eq:6.9} and the fractional chain rule \cite[Lemma 3.4]{GOV} imply that there exists a positive constant $C>0$ such that the estimate
	\[
	\| \mathcal M(f) \|_{\dot B_{P,2}^{s_0}}
	\leq C \| f \|_{\dot B_{P_0,2}^{s_0}}
	\| f \|_{L^{(\rho_0-1) P_1}}^{\rho_0-1}
	\]
holds, provided that the right-hand side is finite. In order to show the estimate \eqref{eq:6.12},
it is sufficient to prove that for the numbers $P_0$ and $P_1$, the embedding holds:
	\begin{align}
	X_0^{s, \alpha} \times X_0^{s', \alpha}
	\hookrightarrow \dot B_{P_0,2}^{s_0} \times L^{(\rho_0-1) P_1}.
	\label{eq:6.14}
	\end{align}
We claim that the relations
	\begin{align}
	X_0^{s, \alpha}
	\hookrightarrow H^{s} \cap L^\gamma
	\hookrightarrow \dot B_{P_0,2}^{s_0}
	\label{eq:6.15}
	\end{align}
hold if $\gamma\in (1,2]$ satisfies the condition \eqref{eq:6.4}
with $\nu = \gamma$ and
	\begin{align}
	\frac {s_0} s \frac 1 2 + \frac{s-s_0}{s} \frac 1 \gamma
	\geq \frac 1 {P_0}
	\geq \bigg( \frac 1 2 - \frac{s-s_0}{n} \bigg)_+.
	\label{eq:6.16}
	\end{align}
The first embedding of \eqref{eq:6.15} follows from Lemma \ref{Lemma:6.1}.
When $\frac 1 {P_0}$ is given by
the LHS of the first estimate of \eqref{eq:6.16},
the second embedding of \eqref{eq:6.15} follows from Lemma \ref{Lemma:4.1}
with $\left(\theta,\dot B_{p_1,q_1}^{s_1},\dot B_{p_2,q_2}^{s_2}\right)$
replaced by $\left(\frac {s_0}{s}, \dot B_{2,2}^s, \dot B_{\gamma,2}^0\right)$.
When $\frac 1 {P_0}$ is given by
the RHS of the last estimate of \eqref{eq:6.16},
the second embedding of \eqref{eq:6.15} follows from the Sobolev embedding.
Here we note that the embedding $\dot B_{2,2}^{\frac n 2} \hookrightarrow \dot B_{\infty,2}^0$
also holds.
Therefore,
we obtain the relations \eqref{eq:6.15}
for $P_0$ satisfying \eqref{eq:6.16}
by the interpolation.
Then the first estimate of \eqref{eq:6.4} and \eqref{eq:6.16} imply that
the estimates
	\begin{align}
	\frac 1 2 + \frac{s-s_0}{s} \frac{\min(2\alpha,n)}{2n}
	&\ge \frac {s_0} s \frac 1 2 + \frac{s-s_0}{s} \min\bigg( \frac{2 \alpha + n}{2n}, 1 \bigg)
	\nonumber\\
	&> \frac 1 {P_0}
	\geq \bigg( \frac 1 2 - \frac{s-s_0}{n} \bigg)_+
	\label{eq:6.17}
	\end{align}
hold.
In addition, if the condition \eqref{eq:6.4} with $(s,\gamma)$ replaced by $(s',(\rho_0-1) P_1)$ holds i.e., the estimates
	\begin{align}
	\frac{2 \alpha + n}{2n}
	> \frac 1 {(\rho_0-1) P_1}
	\begin{cases}
	\geq ( \frac 1 2 - \frac{s'}{n} )_+
	&\mathrm{if} \quad s' \neq \frac n 2,\\
	> 0
	&\mathrm{if} \quad s' = \frac n 2,
	\end{cases}
	\label{eq:6.18}
	\end{align}
hold, then Lemma \ref{Lemma:6.1} implies that
$X_0^{s',\alpha}$ is embedded into $L^{(\rho_0-1) P_1}$.
We remark that \eqref{eq:6.18} may admit $P_1 < 1$
but the assumption $P \geq 1$ excludes this case.
By summing up the first and last estimates of \eqref{eq:6.17} and \eqref{eq:6.18},
we obtain the conditions \eqref{eq:6.10} and \eqref{eq:6.11}, respectively.

We next show the estimate \eqref{eq:6.13}.
When $s_0 = 1$, \eqref{eq:6.13} follows from the classical chain rule of differentiation.
When $s_0\ne 1$,
we divide the proof
into two cases where $s_0 < 1$ and
$s_0 > 1$.

We start from the proof in the case where $s_0 < 1$.
We compute
	\begin{align*}
	&(\tau_y-1) (\mathcal M(f) - \mathcal M(g))\\
	&= \int_0^1 \mathcal M'(\theta \tau_y f + (1-\theta) f) d \theta (\tau_y -1 ) f
	- \int_0^1 \mathcal M'(\theta \tau_y g + (1-\theta) g) d \theta (\tau_y -1 ) g\\
	&= \int_0^1 \mathcal M'(\theta \tau_y f + (1-\theta) f) d \theta (\tau_y -1 ) (f-g)\\
	&+ \int_0^1 ( \mathcal M'(\theta \tau_y f + (1-\theta) f)
		- \mathcal M'(\theta \tau_y g + (1-\theta) g) ) d \theta (\tau_y -1 ) g.
	\end{align*}
Hence, by the assumption (\ref{eq:6.9}), the estimate
	\begin{align}
	&| (\tau_y-1) (\mathcal M(f) - \mathcal M(g))|
	\nonumber\\
	&\leq C (\tau_y+1)|f|^{\rho_0-1} | (\tau_y-1) (f -g)|
	\nonumber\\
	&+ C (\tau_y+1)(|f|+|g|)^{(\rho_0-2)_+}
	(\tau_y+1)|f-g|^{\min(\rho_0-1,1)} | (\tau_y-1) g|
	\label{eq:6.19}
	\end{align}
holds.
By the H\"older inequality with $P_0$ and $P_1$ above,
\eqref{eq:6.19} implies that the estimate
	\begin{align}
	\| (\tau_y-1) (\mathcal M(f) - \mathcal M(g))\|_{L^{P}}
	&\leq C \|f\|_{L^{(\rho_0 -1) P_1}}^{\rho_0-1}
	\| (\tau_y-1) (f -g)\|_{L^{P_0}}
	\nonumber\\
	&+ C \|f\|_{L^{(\rho_0 -1) P_1}}^{(\rho_0-2)_+}
	\|f-g\|_{L^{(\rho_0 -1) P_1}}^{\min(\rho_0-1,1)}
	\| (\tau_y-1) g \|_{L^{P_0}}
	\label{eq:6.20}
	\end{align}
holds.
Noting that the identity
	\[
	(\rho_0 - 2)_+ + \min(\rho_0 -1 ,1)
	= \rho_0 -1
	\]
holds, and combining
\eqref{eq:1.15}, \eqref{eq:6.14}, and \eqref{eq:6.20},
we obtain \eqref{eq:6.13} when $s_0 < 1$.

We next consider the case where $s_0 > 1$.
In this case, we use the equivalence \eqref{eq:1.16}.
Lemma \ref{Lemma:A.1} implies that it is sufficient to show that there exists a positive constant $C_{d_{\rho_0},s_0}>0$ such that the estimate
	\begin{align}
	&\| h^{-s_0} \sup_{|y| \leq h}
		\|
			|f-g|^{d_{\rho_0,s_0}} |(\tau_y-1) f|^{\rho_0-1-d_{\rho_0,s_0}} (\tau_y-1) g
		\|_{L^P}
	\|_{L^2(0,\infty, \frac {dh} h)}
	\nonumber\\
	&\leq C \| f - g \|_{X_0^{s',\alpha}}^{d_{\rho_0,s_0}}
	( \|f\|_{X_0^{s,\alpha}} + \|g\|_{X_0^{s,\alpha}})^{\rho_0 - d_{\rho_0,s_0}}
	\label{eq:6.21}
	\end{align}
holds, since the other terms are easily or similarly controlled.
With $P_0$ and $P_1$ above, we take a positive number $P_2$ satisfying the identity
	\begin{align*}
	\frac 1 P
	&= \frac{d_{\rho_0,s_0}}{(\rho_0-1) P_1} + \frac 1 {P_2},
	\end{align*}
holds, namely, we set
	\[
	\frac 1 P_2
	:= \frac 1 {P_0} + \frac{\rho_0-d_{\rho_0,s_0}-1}{(\rho_0-1)P_1}.
	\]
Then by the H\"older inequality, the following estimate holds:	\begin{align}
	&\| |f-g|^{d_{\rho_0,s_0}} |(\tau_y-1) f|^{\rho_0-d_{\rho_0,s_0}-1} (\tau_y-1) g
		\|_{L^P}
	\nonumber\\
	&\leq
	\| f - g \|_{L^{(\rho_0-1) P_1}}^{d_{\rho_0,s_0}}
	\| (\tau_y - 1) f \|_{L^{(\rho_0-d_{\rho_0,s_0}) P_2}}^{\rho_0-d_{\rho_0,s_0}-1}
	\| (\tau_y - 1) g \|_{L^{(\rho_0-d_{\rho_0,s_0}) P_2}}.
	\label{eq:6.22}
	\end{align}
Here we note that the $\dot B_{(\rho_0 - d_{\rho_0,s_0}) P_2, 2 (\rho_0 - d_{\rho_0,s_0})}^{\frac{s_0}{\rho_0-d_{\rho_0,s_0}}}$ norm allows the equivalence of \eqref{eq:1.15}
because of the assumption $d_{\rho_0,s_0} < \rho_0 - s_0$.
Therefore, by the estimate \eqref{eq:6.22}, the H\"older inequality,
and \eqref{eq:6.14}, the LHS of \eqref{eq:6.21}
is estimated by	\begin{align}
	&\| h^{-s_0} \sup_{|y| \leq h} \| |f-g|^{d_{\rho_0,s_0}} |(\tau_y-1) f|^{\rho_0-d_{\rho_0,s_0}-1}
		(\tau_y-1) g\|_{L^P} \|_{L^2(0,\infty; \frac {dh}h)}
	\nonumber\\
	&\leq \| f-g \|_{L^{(\rho_0-1)P_1}}^{d_{\rho_0,s_0}}
	\nonumber\\
	&\cdot \| h^{-s_0} \sup_{|y| \leq h}
		(\|(\tau_y-1) f\|_{L^{(\rho_0-d_{\rho_0,s_0})P_2}}^{\rho_0-d_{\rho_0,s_0}-1}
		\| (\tau_y-1) g \|_{L^{(\rho_0-d_{\rho_0,s_0})P_2}})
		\|_{L^2(0,\infty; \frac {dh}h)}
	\nonumber\\
	&\leq \| f - g \|_{X_0^{s',\alpha}}^{d_{\rho_0,s_0}}
	\| f
		\|_{\dot B_{(\rho_0-d_{\rho_0,s_0}) P_2, 2(\rho_0-d_{\rho_0,s_0})}%
		^{\frac{s_0}{\rho_0 -d_{\rho_0,s_0}}}}^{\rho_0-d_{\rho_0,s_0}-1}
	\| g
		\|_{\dot B_{(\rho_0-d_{\rho_0,s_0}) P_2, 2(\rho_0-d_{\rho_0,s_0})}%
		^{\frac{s_0}{\rho_0-d_{p_0,s_0}}}}.
	\label{eq:6.23}
	\end{align}
The interpolation estimate Lemma \ref{Lemma:4.1} and \eqref{eq:6.14} imply that
the estimates
	\begin{align}
	\| f
		\|_{\dot B_{(\rho_0-d_{\rho_0,s_0}) P_2, 2(\rho_0-d_{\rho_0,s_0})}%
		^{\frac{s_0}{\rho_0-d_{\rho_0,s_0}}}}
	\leq C
	\| f \|_{\dot B_{P_0, 2}^{s_0}}^{\frac 1 {\rho_0-d_{\rho_0,s_0}}}
	\| f \|_{\dot B_{(\rho_0-1)P_1, \infty}^{0}}^{\frac{\rho_0-d_{\rho_0,s_0}-1}{\rho_0-d_{\rho_0,s_0}}}
	\leq C \| f \|_{X_0^{s,\alpha}}
	\label{eq:6.24}
	\end{align}
hold.
Hence, the estimates \eqref{eq:6.23} and \eqref{eq:6.24} imply that the inequality \eqref{eq:6.21} holds, which completes the proof of the lemma.

\end{proof}

Now we show the following nonlinear estimates from the solution space $X_0^{s,\alpha}$ to the auxiliary space $Y_0^{s,\alpha,\gamma}$:
\begin{Lemma}[Estimate for the nonlinearity from the solution space to the auxiliary space]
\label{Lemma:6.5}
Let $n\in \mathbb{N}$, $\alpha \geq 0$, $0 \leq s' \leq s < \rho < \infty$ and $\gamma \in [1,2]$.
We assume that the estimates	\begin{align}
	\frac{2n}{2\alpha+n}
	< \rho \gamma
	\begin{cases}
	< \infty
	& \mathrm{if} \quad s \geq n/2,\\
	\leq \frac{2n}{n-2s}
	& \mathrm{if} \quad s < n/2
	\end{cases}
	\label{eq:6.25}
	\end{align}
hold and $\rho$ satisfies the inequalities
	\begin{align}
	1 + \frac{\min(n,\sigma)}{2 \alpha + n}
	< \rho \leq 1 + \frac{\min(n,\sigma)}{(n-2s')_+}.
	\label{eq:6.26}
	\end{align}
Then there exists a positive constant $C>0$ such that the estimates
	\begin{align*}
	\| \mathcal N(f) \|_{Y_0^{s,\alpha,\gamma}}
	&\leq C \| f \|_{X_0^{s,\alpha}}^\rho,\\
	\| \mathcal N(f) - \mathcal N(g) \|_{Y_0^{s,\alpha,\gamma}}
	&\leq C (\| f \|_{X_0^{s,\alpha}} + \| g \|_{X_0^{s,\alpha}})^{\rho-1}
	\| f - g \|_{X_0^{s,\alpha}}\\
	&+ C (\| f \|_{X_0^{s,\alpha}} + \| g \|_{X_0^{s,\alpha}})^{\rho-d}
	\| f - g \|_{X_0^{s',\alpha}}^d
	\end{align*}
hold for any $f, g \in X_0^{s,\alpha}$
with a positive number $d \leq 1$ satisfying
	\begin{align}
	d
	\begin{cases}
	= \min(\rho-[s]-1,1),
	&\mathrm{if}
	\quad s \not\in \mathbb Z
	\ \mathrm{and} \ \rho - [s] > 1,\\
	< \rho - s,
	&\mathrm{otherwise}.
	\end{cases}
	\label{eq:6.27}
	\end{align}
\end{Lemma}

\begin{proof}
At first we estimate the homogeneous Besov $\dot B_{q_\sigma,2}^{s}$-norm of the nonlinearity $\mathcal N$ with $s>0$, i.e. we prove that there exists a positive constant $C>0$ such that the estimates
	\begin{align}
	\| \mathcal N(f) \|_{\dot B_{q_\sigma,2}^{s}} \leq C \| f \|_{X_0^{s,\alpha}}^\rho
	\label{eq:6.28}
	\end{align}
and
	\begin{align}
	\| \mathcal N(f) - \mathcal N(g) \|_{\dot B_{q_\sigma,2}^{s}}
	&\leq C (\| f \|_{X_0^{s,\alpha}} + \| g \|_{X_0^{s,\alpha}})^{\rho-1}
	\| f - g \|_{X_0^{s,\alpha}}
	\nonumber\\
	&+ C (\| f \|_{X_0^{s,\alpha}} + \| g \|_{X_0^{s,\alpha}})^{\rho-d}
	\| f - g \|_{X_0^{s',\alpha}}^d
	\label{eq:6.29}
	\end{align}
hold for any $f,g\in X^{s,\alpha}_0$. We note that the identities hold:
	\begin{align}
	\frac{\min(n,\sigma)}{2 n} + \frac 1 2
	= \frac {n+\sigma}{n+\max(n,\sigma)} = \frac 1 {q_\sigma}.
	\label{eq:6.30}
	\end{align}
We divide the proof into the two cases where
$0 < s < \min(2,\rho)$ and $2 \leq s < \rho$.

When $0 < s < \min(2,\rho)$,
the estimates \eqref{eq:6.28} and \eqref{eq:6.29}
follow directly from Lemma \ref{Lemma:6.3}, since \eqref{eq:6.10} and \eqref{eq:6.11} hold
with $(P, \mathcal M, \rho_0, s_0)$
replaced by $(q_\sigma, \mathcal N, \rho, s)$.
Indeed, by the identities \eqref{eq:6.30}, the conditions \eqref{eq:6.10} and \eqref{eq:6.11} can be written as
	\[
	\frac{2\alpha + n}{2n} (\rho-1) + \frac 1 2
	> \frac 1 {q_\sigma}
	= \frac 1 2 + \frac{\min(n,\sigma)}{2 n}
	\geq \frac 1 2 + \frac{(n - 2 s')_+}{2n}(\rho-1),
	\]
which is equivalent to the condition \eqref{eq:6.26} when $s' \neq \frac n 2$.
When $s' = \frac n 2$,
no upper bound for $\rho$ is required here.

Next we assume $2 \leq s < \rho$.
We remark that
under this assumption,
the first estimate of \eqref{eq:6.26} holds,
since
	\[
	1+ \frac{\min(n,\sigma)}{2 \alpha + n}
	\leq 2 \leq s < \rho.
	\]
Let
	\[
	s_0:=
	\begin{cases}
	s- [s]
	&\mathrm{if} \quad s \not\in \mathbb Z \ \mathrm{and} \ \rho - [s] > 1,\\
	1
	&\mathrm{if} \quad s \in \mathbb Z,\\
	s- [s] + 1
	&\mathrm{if} \quad s \not\in \mathbb Z \ \mathrm{and} \ \rho - [s] \leq 1.
	\end{cases}
	\]
Then we note that
	\[
	s_0
	\begin{cases}
	< 1
	&\mathrm{if} \quad s \not\in \mathbb Z \ \mathrm{and} \ \rho - [s] > 1,\\
	=1
	&\mathrm{if} \quad s \in \mathbb Z,\\
	\in (1, 2)
	&\mathrm{if} \quad s \not\in \mathbb Z \ \mathrm{and} \ \rho - [s] \leq 1.
	\end{cases}
	\]
We claim that there exists a positive constant $C>0$ such that the estimate
    \begin{align}
    \| h \|_{\dot B_{q_\sigma,2}^s}
    \leq  C \sup_{|\alpha|=s-s_0} \| \partial^\alpha h \|_{\dot B_{q_\sigma,2}^{s_0}}.
    \label{eq:6.31}
    \end{align}
holds provided that the right hand side is finite. Indeed, for any $s_1 > 0$,
$(-\Delta)^{s_1/2}$ is isomorphism
from $\dot B_{q_\sigma,2}^s$ to $\dot B_{q_\sigma,2}^{s-s_1}$.
Moreover, $\partial^\alpha$ is a continuous map
from $\dot B_{q_\sigma,2}^{s-s_1+|\alpha|}$ to $\dot B_{q_\sigma,2}^{s-s_1}$
for any multi index $\alpha$.
Thus we see that \eqref{eq:6.31} holds.
For positive integers $j$ and $m$,
we set
	\begin{align*}
	\mathcal K_{m}^j
	:= \{ \kappa \in (\mathbb Z_{\geq 1})^{j}, |\kappa| = m\}.
	\end{align*}
Then the classical Leibniz rule and Lemma \ref{Lemma:4.2} imply that the estimate
	\begin{align*}
	\| \mathcal N (f) \|_{\dot B_{q_\sigma,2}^{s}}
	&\leq C
	\sum_{\overline s = 1}^{s-s_0}
	\sum_{\kappa \in \mathcal K_{s-s_0}^{\overline s}}
	 \| \mathcal N^{(\overline s)} (f) \|_{\dot B_{Q_0,2}^{s_0}}
	\prod_{j=1}^{\overline s} \| D^{\kappa_j} f \|_{L^{Q_j}}\\
	&+ C
	\| \mathcal N^{(1)} (f) \|_{L^{\widetilde Q_0}}
	\| f \|_{\dot B_{\widetilde Q_{1},2}^{s}}\\
	&+ C \sum_{\overline s = 2}^{s-s_0}
	\sum_{\kappa \in \mathcal K_{s-s_0}^{\overline s}}
	\sum_{j_0=1}^{\overline s}
	\| \mathcal N^{(\overline s)} (f) \|_{L^{\widetilde Q_0}}
	\| f \|_{\dot B_{\widetilde Q_{j_0},2}^{\kappa_{j_0}+s_0}}
	\prod_{j\in [1,\overline s] \backslash \{j_0\}}
	\| D^{\kappa_j} f \|_{L^{\widetilde Q_j}}
	\end{align*}
holds for
$Q=(Q_j)_{j=0}^{\overline s}=(Q_j(\overline s,\kappa))_{j=0}^{\overline s}$
and
$\widetilde Q=\left(\widetilde Q_j\right)_{j=0}^{\overline s}
= \left(\widetilde Q_j(\overline s,\kappa,j_0)\right)_{j=0}^{\overline s}$
satisfying the H\"older conditions
	\begin{align}
	\sum_{j=0}^{\overline s} \frac 1 {Q_j}
	&= \frac 1 {q_\sigma},
	\label{eq:6.32}\\
	\sum_{j=0}^{\overline s} \frac 1 {\widetilde Q_j}
	&= \frac 1 {q_\sigma}.
	\label{eq:6.33}
	\end{align}
We shall choose a proper pair $(Q,\widetilde Q)$
so that \eqref{eq:6.28} and \eqref{eq:6.29} hold.
We divide the proof
into two cases where $s \leq \frac n 2$ and $s > \frac n 2$.

We consider the case where $2 \leq s \leq \frac n 2$.
We remark that $s' \leq \frac n 2$ holds in this case.
We choose $Q$ with \eqref{eq:6.32} and under the following conditions:
	\begin{align*}
	\frac {\kappa_j}{n} + \frac{n-2s}{2n}
	&\leq \frac 1 {Q_j}
	\leq \frac 1 2
	\quad \mathrm{for} \ j \in [1,\overline s],\\
	\bigg( \frac 1 2 + \frac{\min(n,\sigma)}{2n}
	- \frac{\overline s}{2} \bigg)_+
	&\leq \frac{1}{Q_0} 
	\leq \frac 1 2 + \frac{\min(n,\sigma)}{2n}
	- \frac{s-s_0}{n} - \overline s \frac{n-2s}{2n}.
	\end{align*}
Here we note that
$1 \leq \overline s \leq s-s_0$ and
	\[
	\sum_{j=1}^{\overline s} \kappa_j = s-s_0.
	\]
We claim that
under the condition above,
we can take $Q_0$ satisyfying
\eqref{eq:6.10} and \eqref{eq:6.11} with $(P,\rho_0)$ replaced by $(Q_0,\rho-\overline s)$
for any $1 \leq \overline s \leq s-s_0$.
Indeed, we have
	\begin{align*}
	&\frac{2 \alpha + n}{2n} (\rho - \overline s -1) + \frac 1 2 + \frac{s - s_0}{s} \frac{\min(2\alpha,n)}{2n}
	- \bigg( \frac 1 2 + \frac{\min(n,\sigma)}{2n}
	- \frac{\overline s}{2} \bigg)_+\\
	&> \min\bigg(
	\frac{\overline s}{2} - \frac{\min(n,\sigma)}{2n}, \frac 1 2 \bigg)
	\geq 0,
	\end{align*}
where we have used the facts that
$\rho-\overline s \geq \rho-s+s_0 > 1$
and $s_0 \leq s-[s]+1$.
Here $\rho-s+s_0-1$ is computed and estimated as
	\[
	\rho-s+s_0 -1
	= \begin{cases}
	\rho-[s] -1
	&\mathrm{if} \quad s \in \mathbb Z \quad \mathrm{and} \quad \rho - [s] > 1,\\
	\rho - s
	&\mathrm{if} \quad s \in \mathbb Z,\\
	\rho-[s] \geq \rho - s
	&\mathrm{if} \quad s \in \mathbb Z \quad \mathrm{and} \quad \rho - [s] \leq 1.
	\end{cases}
	\]
Moreover, the last estimate of \eqref{eq:6.26} implies that we have
	\begin{align*}
	&\bigg(\frac 1 2 + \frac{\min(n,\sigma)}{2 n} - \frac {s-s_0} n
	- \overline s \bigg( \frac 1 2 - \frac s n \bigg) \bigg)
	- \bigg( \frac 1 2 - \frac {s-s_0}{n}
	+ (\rho-\overline s - 1) \bigg( \frac 1 2 - \frac{s'} n \bigg) \bigg)\\
	&\geq \frac{\min(n,\sigma)}{2n} - (\rho-1) \bigg( \frac 1 2 - \frac{s'} n \bigg)
	\ge 0.
	\end{align*}
These two estimates imply that
we can take $Q_0$ satisfying
\eqref{eq:6.10} and \eqref{eq:6.11} with $(P,\rho_0)$ replaced by $(Q_0,\rho-\overline s)$
for any $1 \leq \overline s \leq s-s_0$ hold.
Therefore, Lemma \ref{Lemma:6.3} implies that the estimates
	\begin{align}
	\| \mathcal N^{(\overline s)} (f) \|_{\dot B_{Q_0,2}^{s_0}}
	&\leq \| f \|_{X_0^{s,\alpha}}^{\rho-\overline s},
	\label{eq:6.34}\\
	\| \mathcal N^{(\overline s)} (f) - \mathcal N^{(\overline s)} (g) \|_{\dot B_{Q_0,2}^{s_0}}
	&\leq (\| f \|_{X_0^{s,\alpha}} + \| g \|_{X_0^{s,\alpha}} )^{\rho-\overline s-1}
	\| f- g\|_{X_0^{s,\alpha}}
	\nonumber\\
	&+ (\| f \|_{X_0^{s,\alpha}} + \| g \|_{X_0^{s,\alpha}}
		)^{\rho-\overline s-d_{\rho-\overline s,s_0}}
	\| f- g\|_{X_0^{s',\alpha}}^{d_{\rho-\overline s,s_0}}
	\label{eq:6.35}
	\end{align}
hold.
Here we note that
\eqref{eq:6.27} implies that we have
	\begin{align*}
	\min_{1 \leq \overline s \leq s-s_0} d_{\rho-\overline s, s_0}
	&= d_{\rho-s + s_0,s_0}
	\nonumber\\
	&\begin{cases}
		= \min(\rho-[s]-1, 1)
		&\mathrm{if} \quad s \not\in \mathbb Z \ \mathrm{or} \ \rho - [s] > 1,\\
		< \rho-s
		&\mathrm{otherwise}
	\end{cases}
	\nonumber\\
	&=d.
	\end{align*}
Moreover
the Sobolev inequality implies that
the relations
	\begin{align}
	X_0^{s,\alpha} \hookrightarrow H^s \hookrightarrow \bigcap_{j \in [1,\overline s]}H_{Q_j}^{\kappa_j}
	\label{eq:6.36}
	\end{align}
hold for any $j \in [1,\overline s]$.
We next choose $\widetilde Q$ by
	\begin{align*}
	\frac 1 {\widetilde Q_0}
	&= \frac{\min(n,\sigma)}{2 n}
	- (\overline s-1) \frac{(n-2s)}{2n},\\
	\frac 1 {\widetilde Q_{j_0}}
	&= \frac {\kappa_{j_0}}{n} + \frac{n-2(s-s_0)}{2n},\\
	\frac 1 {\widetilde Q_j}
	&= \frac {\kappa_j}{n} + \frac{n-2s}{2n}
	\quad \mathrm{for} \ j \in [1,\overline s] \backslash \{j_0\}.
	\end{align*}
Then $\widetilde Q$ satisfies \eqref{eq:6.33}.
The relation $X_0^{s',\alpha} \hookrightarrow L^{(\rho-\overline s) \widetilde Q_0}$ also holds.
Indeed the estimates
	\[
	(\rho-\overline s) \widetilde Q_0
	> 2 \geq \frac{2n}{2 \alpha + n}
	\]
hold since $\rho - \overline s > 1$ and $\widetilde Q_0 \geq 2$.
Moreover, the last estimate of \eqref{eq:6.26} implies that $\widetilde Q_0$ is estimated by
	\[
	\frac{1}{\widetilde Q_0}
	\geq \frac{\min(n,\sigma)}{2 n}
	- (\rho-1) \frac{(n-2s')}{2n}
	+ (\rho-\overline s) \frac{(n-2s')}{2n}
	\geq (\rho-\overline s) \frac{(n-2s')}{2n}.
	\]
Therefore, the estimate
	\begin{align*}
	(\rho-\overline s) \widetilde Q_0
	&\leq
	\begin{cases}
	\frac{2n}{n-2s'}
	&\mathrm{if} \quad s' < \frac n 2,\\
	(\rho-\overline s)\frac{\min(n,\sigma)}{2n} < \infty
	&\mathrm{if} \quad s' = \frac n 2
	\end{cases}
	\end{align*}
holds.
These estimates above and Lemma \ref{Lemma:6.3} imply that the estimates
	\begin{align}
	\| \mathcal N^{(\overline s)} (f) \|_{L^{\widetilde Q_0}}
	&\leq C\| f \|_{X_0^{s,\alpha}}^{\rho-\overline s},
	\label{eq:6.37}\\
	\| \mathcal N^{(\overline s)} (f) - \mathcal N^{(\overline s)} (g) \|_{L^{\widetilde Q_0}}
	&\leq C(\| f \|_{X_0^{s,\alpha}} + \| g \|_{X_0^{s,\alpha}} )^{\rho-\overline s-1}
	\| f- g\|_{X_0^{s,\alpha}}
	\nonumber\\
	&+ C(\| f \|_{X_0^{s,\alpha}} + \| g \|_{X_0^{s,\alpha}}
		)^{\rho-\overline s-d_{\rho-\overline s,s_0}}
	\| f- g\|_{X_0^{s',\alpha}}^{d_{\rho-\overline s,s_0}}
	\label{eq:6.38}
	\end{align}
hold.
It is also similarly seen that
$\widetilde Q$ satisfies the other requirements.
So we have
	\begin{align}
	X_0^{s,\alpha}
	\hookrightarrow \dot H_{\widetilde Q_{j_0}}^{\kappa_{j_0}+s_0}
	\cap \bigcap_{j \neq j_0} \dot H_{\widetilde Q_{j}}^{\kappa_{j}}.
	\label{eq:6.39}
	\end{align}
Finally, estimates \eqref{eq:6.28} and \eqref{eq:6.29}
follow from \eqref{eq:6.34}--\eqref{eq:6.39}.

When $s > \frac n 2$,
we choose $Q_j$ with \eqref{eq:6.32} under the following condition:
	\begin{align*}
	\frac{\kappa_j}{n}
	+ \frac{\kappa_j}{s-s_0} \bigg( \frac 1 2 - \frac{s-s_0}{n} \bigg)_-
	&\leq \frac 1 {Q_j}
	\leq \frac 1 2
	\quad \mathrm{for} \ j \in [1,\overline s],\\
	\bigg( \frac 1 2 + \frac{\min(n,\sigma)}{2n}
	- \frac{\overline s}{2} \bigg)_+
	&\leq \frac{1}{Q_0}
	\leq \bigg( \frac 1 2 - \frac{s - s_0}{n} \bigg)_+
	+ \frac{\min(n,\sigma)}{2n}.
	\end{align*}
It is seen that
the following estimates hold for $j \in [1, \overline s]$:
	\begin{align*}
	&\frac {\kappa_j}{n} + \bigg( \frac{\kappa_j}{2(s-s_0)}
	- \frac {\kappa_j} n \bigg)_-
	\leq \frac{\kappa_j}{2(s-s_0)} \leq \frac 1 2,\\
	& - \frac{\kappa_j}{s-s_0} \bigg( \frac n 2 - s + s_0 \bigg)_-
	= \frac{\kappa_j}{s-s_0} \bigg( s - \frac n 2 - s_0 \bigg)_+
	\leq s - \frac n 2.
	\end{align*}
Therefore, \eqref{eq:6.36} holds for any $Q$ under the condition above.
Moreover,
the second estimate of \eqref{eq:6.26} implies that
	\begin{align*}
	&\bigg( \frac 1 2 - \frac{s - s_0}{n} \bigg)_+
	+ \frac{\min(n,\sigma)}{2n}
	- \bigg( \bigg( \frac 1 2 - \frac {s-s_0}{n} \bigg)_+
	+ (\rho-\overline s - 1) \bigg( \frac 1 2 - \frac{s'} n \bigg)_{+} \bigg)\\
	&> 0.
	\end{align*}
Therefore, \eqref{eq:6.34} and \eqref{eq:6.35} hold.

We also choose
	\begin{align*}
	\frac 1 {\widetilde Q_0}
	&= \frac{\min(n,\sigma)}{2n},\\
	\frac 1 {\widetilde Q_{j}}
	&= 
	\frac{\kappa_j}{n}
	+\frac{\kappa_j}{s-s_0}
	\left(\frac{1}{2}
	    -\frac{s-s_0}{n}
	\right)_{-}
		\quad \mathrm{for} \ j \in [1,\overline s] \backslash \{ j_0 \},\\
	\frac 1 {\widetilde Q_{j_0}}
		&= \frac {\kappa_{j_0}}{n}
		+ \frac{\kappa_{j_0}}{s-s_0} \bigg( \frac 1 2 - \frac {s-s_0} n \bigg)_-
		+ \bigg( \frac 1 2 - \frac {s-s_0} n \bigg)_+.
	\end{align*}
Then $\widetilde{Q}$ satisfies
\eqref{eq:6.33}.
The estimates
	\begin{align*}
	(\rho-\overline s) \widetilde Q_0
	&\geq \frac{2 n}{\min(n,\sigma)} \geq 2,\\
	(\rho-\overline s) \widetilde Q_0
	&\leq (\rho-1)\frac{2 n}{\min(n,\sigma)}
	\begin{cases}
	\leq \frac{2 n}{(n-2s')_+}
	&\mathrm{if} \quad s' < \frac n 2,\\
	< \infty
	&\mathrm{if} \quad s' \geq \frac n 2
	\end{cases}
	\end{align*}
follow from the estimate $\rho-\overline s \geq 1$ and \eqref{eq:6.26}.
This implies that
\eqref{eq:6.37} and \eqref{eq:6.38} hold.
Besides $H^{s} \hookrightarrow  H^{\kappa_{j_0}+s_0}_{\widetilde Q_{j_0}}$,
since
	\begin{align*}
	\frac 1 {\widetilde Q_{j_0}}
	&\leq
	\begin{cases}
	\frac {\kappa_{j_0}}{n} + \frac{\kappa_{j_0}}{2(s-s_0)} - \frac {\kappa_{j_0}}{n}
	\leq  \frac 1 2
	&\mathrm{if} \quad s-s_0 \geq \frac n 2,\\
	\frac {\kappa_{j_0}}{n} + \frac 1 2 - \frac {s-s_0}{n}
	\leq \frac 1 2
	&\mathrm{if} \quad s-s_0 \leq \frac n 2,
	\end{cases}\\
	s_0 + \kappa_{j_0} - \frac n {\widetilde Q_{j_0}}
	&= s_0 - \frac{\kappa_{j_0}}{s-s_0} \bigg( \frac n 2 - s + s_0 \bigg)_-
	- \bigg( \frac n 2 - s +s_0 \bigg)_+\\
	&\leq s_0 + \bigg( s - \frac n 2 - s_0 \bigg)_+
	+ \bigg( s - \frac n 2 - s_0 \bigg)_-
	\leq s - \frac n 2.
	\end{align*}
From this, we can see that \eqref{eq:6.39} holds,
since the conditions for
$\widetilde{Q}_j$ with $j\neq j_0$
are easily verified.
Therefore,
the properties
\eqref{eq:6.34}--\eqref{eq:6.39}
holds in the same way as the previous case,
and hence, we obtain \eqref{eq:6.28} and \eqref{eq:6.29}.

Next, we estimate the Lebesgue $L^{\gamma}$ norms of the nonlinearity $\mathcal{N}(f)$ without derivative or weight.
Since \eqref{eq:6.25} implies that $\gamma \rho$ satisfies \eqref{eq:6.4} with $\nu$ replaced by $\gamma \rho$,
the argument of Lemma \ref{Lemma:6.1} implies that the estimates
	\begin{align*}
	\|\mathcal N(f)\|_{L^\gamma}
	&\leq C \| f \|_{L^{\rho\gamma}}^\rho \leq C \| f \|_{X_0^{s,\alpha}}^p,\\
	\|\mathcal N(f-g)\|_{L^\gamma}
	&\leq C \| (|f|+|g|)^{\rho-1} |f-g| \|_{L^{\gamma}}\\
	&\leq C (\| f \|_{L^{\rho\gamma}}+\| g \|_{L^{\rho\gamma}})^{\rho-1}
	\| f - g\|_{L^{\rho \gamma}}\\
	&\leq C (\| f \|_{X_0^{s,\alpha}}+\| g \|_{X_0^{s,\alpha}})^{\rho-1}
	\| f - g\|_{X_0^{s,\alpha}}
	\end{align*}
hold.
Finally, we note that
we can apply Lemma \ref{Lemma:6.2} because
the assumption \eqref{eq:6.26} implies that
$\rho$ satisfies \eqref{eq:2.5} and \eqref{eq:6.5}.

Therefore, combining these estimates above,
the proof is completed.
\end{proof}

\begin{Corollary}
\label{Corollary:6.6}
Let $s \geq 0$ and $T>0$.
We assume the same assumptions of Lemma \ref{Lemma:6.2} and \eqref{eq:6.25} when $s=0$
and those of Lemma \ref{Lemma:6.5} when $s > 0$.
Then there exists a positive constant $C>0$ independent of $T$ such that the estimates
	\begin{align}
	\| \mathcal N(u) \|_{Y^{s,\alpha,\gamma,r}(T)}
	&\leq C \| u \|_{X^{s,\alpha,r}(T)}^\rho,
	\label{eq:6.40}\\
	\| \mathcal N(u) - \mathcal N(v) \|_{Y^{s,\alpha,\gamma,r}(T)}
	&\leq C (\| u \|_{X^{s,\alpha,r}(T)} + \| v \|_{X^{s,\alpha,r}(T)})^{\rho-1}
	\| u - v \|_{X^{s,\alpha,r}(T)}
	\nonumber\\
	&+ C (\| u \|_{X^{s,\alpha,r}(T)} + \| v \|_{X^{s,\alpha,r}(T)})^{\rho-d}
	\| u - v \|_{X^{s',\alpha,r}(T)}^d
	\label{eq:6.41}
	\end{align}
hold for any $u,v \in X^{s,\alpha,r}(T)$,
where $d$ is given by \eqref{eq:6.27}.
\end{Corollary}

\begin{proof}
By Lemmas \ref{Lemma:6.2} and \ref{Lemma:6.5},
for any $t > 0$, the estimates
	\[
	\| \mathcal N(u) \|_{Y^{s,\alpha,\gamma,r}(T)}
	= \langle t \rangle^{\frac{n\rho}{r \sigma}}
	\| \mathcal N( \langle t \rangle^{- \frac{n}{r \sigma}} \mathcal D_{\sigma,r}(t) u) \|_{Y_0^{s,\alpha,\gamma}}
	\leq C \| \mathcal D_{\sigma,r}(t) u \|_{X_0^{s,\alpha}}^\rho
	\]
hold. Therefore, \eqref{eq:6.40} and \eqref{eq:6.41} hold.
\end{proof}


\section{Proofs of Theorems \ref{Theorem:2.5} and \ref{Theorem:2.7}}
\subsection{Proof of Theorem \ref{Theorem:2.5}}
Under the assumptions \eqref{eq:2.4}, \eqref{eq:2.5}, and \eqref{eq:2.6},
we claim that for any $r$ satisfying \eqref{eq:5.3},
there exists $\gamma \in [1,\min(r,q_\sigma)]$ satisfying
\eqref{eq:5.4} and \eqref{eq:6.25}.
We prove this claim later.
Then Theorems \ref{Theorem:2.1} and \ref{Theorem:2.3} imply that the estimate
	\begin{align}
	\| \widetilde{\mathcal S_\sigma}(t) u_0
	+ \mathcal S_\sigma (t) u_1 \|_{X^{s,\alpha,\gamma,r}(T)}
	\leq C \| u_0 \|_{X_0^{s,\alpha}}
	+ C \| u_1 \|_{X_0^{s-\frac{\sigma}{2},\alpha}}
	\label{eq:7.1}
	\end{align}
holds.
Moreover, Lemma \ref{Lemma:5.1} and Corollary \ref{Corollary:6.6} imply that
the estimates
	\begin{align}
	&\bigg\| \int_0^t \mathcal S_\sigma (t-\tau) \mathcal N (u)(\tau) d \tau
	\bigg\|_{X^{s,\alpha,r}(T)}
	\leq C T \| \mathcal N (u) \|_{Y^{s,\alpha,\gamma,r}(T)}
	\leq C T \| u \|_{X^{s,\alpha,r}(T)}^\rho,
	\label{eq:7.2}\\
	&\bigg\| \int_0^t \mathcal S_\sigma (t-\tau)
		(\mathcal N (u)(\tau) - \mathcal N (v)(\tau)) d \tau
	\bigg\|_{X^{s,\alpha,r}(T)}
	\nonumber\\
	&\leq C T (\| u \|_{X^{s,\alpha,r}(T)} + \| v \|_{X^{s,\alpha,r}(T)})^{\rho-d}
	\| u - v \|_{X^{s,\alpha,r}(T)}^d
	\label{eq:7.3}
	\end{align}
hold with $d$ defined by \eqref{eq:6.27}.
Now we prove existence of the exponent $\gamma$.
We divide the proof into two cases,
where $\sigma \geq n$ and $n > \sigma$.
In the case of $\sigma\geq n$, we set $\gamma :=1$.
\eqref{eq:2.4} implies that
$\gamma =1$ satisfies \eqref{eq:5.4}, i.e.
	\[
	\frac{rn}{n+r\sigma}
	< 1
	<
	\begin{cases}
	\frac{n}{( \alpha + \frac n 2 - \sigma)_+}
	&\mathrm{if} \quad \sigma \not\in 2 \mathbb Z,\\
	\infty
	&\mathrm{if} \quad \sigma \in 2 \mathbb Z.
	\end{cases}
	\]
Moreover, the estimates
	\begin{align*}
	\frac{2n}{2 \alpha+n}
	&< 1 + \frac{n}{2\alpha+n},\\
	1+\frac{n}{(n-2s)+}
	&\leq \frac{2n}{(n-2s)+}
	\end{align*}
imply that \eqref{eq:6.25} follows from \eqref{eq:2.5} and \eqref{eq:2.6}.
In the case of $n>\sigma$, we choose $\gamma \geq 1$ so that
$\gamma$ satisfies \eqref{eq:5.4} and
	\begin{align}
	\frac{2n}{2 \alpha +n + \sigma}
	\leq \gamma
	< \frac{2n}{(n-2s)_+ + \sigma}.
	\label{eq:7.4}
	\end{align}
We remark that the estimates
	\begin{align*}
	\frac{2n}{(n-2s)_+ + \sigma}
	&> \frac{2 n}{n+2\sigma}
	\geq \frac{nr}{n+r\sigma},\\
	\frac{2n}{(n-2s)_+ + \sigma}
	&> \frac{2n}{(n-2s)_+ + n}
	\geq 1
	\end{align*}
hold. Then \eqref{eq:2.5}, \eqref{eq:2.6}, and \eqref{eq:7.4} imply that the estimates
	\begin{align*}
	\rho \gamma
	& > \frac{2n}{2\alpha+n+\sigma} \left( 1 + \frac{\sigma}{2\alpha+n}\right)
	= \frac{2n}{2\alpha+n}\\
	\rho \gamma
	&< \frac{2n}{(n-2s)_+ + \sigma}\left( 1 + \frac{\sigma}{(n-2s')_+}\right)
	\leq \frac{2n}{(n-2s)_+}
	\end{align*}
hold and imply \eqref{eq:6.25}.

In the case where $s \in \mathbb Z$ and $\rho-s > 1$
and in the case where $s \not\in \mathbb Z$ and $\rho-[s] \geq 2$,
\eqref{eq:7.1}, \eqref{eq:7.2}, and \eqref{eq:7.3} with $d=1$
imply that a unique mild solution to \eqref{eq:1.17} is constructed
by the standard contraction argument on $X^{s,\alpha,r}(T) $
for sufficiently small $T$.

In the cases
where $s \in \mathbb Z$ and $\rho-s \leq 1$
and where $s \not\in \mathbb Z$ and $\rho-[s] < 2$,
\eqref{eq:7.3} with $d=1$ may not hold.
However, by a similar argument,
the estimates
	\begin{align*}
	&\bigg\| \int_0^t \mathcal S_\sigma (t-\tau)
		(\mathcal N (u)(\tau) - \mathcal N (v)(\tau)) d \tau
	\bigg\|_{X^{0,\alpha,r}(T)}\\
	&\leq C \int_0^T \| \langle x \rangle^\alpha \mathcal (N (u)(\tau) - \mathcal N (v)(\tau)) \|_{L^{q_\sigma}} d \tau\\
	&\leq C T (\| u \|_{X^{s,\alpha,r}(T)} + \| u \|_{X^{s,\alpha,r}(T)})^{\rho-1}
	\| u - v \|_{X^{0,\alpha,r}(T)}
	\end{align*}
follow from the assumption $\rho > 1$ and Lemma \ref{Lemma:6.2}.
Since the unit ball on $X^{s,\alpha,r}(T)$ is closed in $X^{0,\alpha,r}(T)$, a unique mild solution to \eqref{eq:1.17} is constructed
by the standard contraction argument in $X^{0,\alpha,r}(T)$.
The continuous dependence of initial data in $X^{s,\alpha,r}(T)$
also follows from Corollary \ref{Corollary:6.6}.
Indeed,
the Gagliardo-Nirenberg inequality implies that
	\[
	\| u - v \|_{X^{s',\alpha,r}(T)}
	\leq C \| u - v \|_{X^{0,\alpha,r}(T)}^{1-\frac {s'}{s}}
	\| u - v \|_{X^{s,\alpha,r}(T)}^{\frac {s'}{s}}.
	\]
Therefore, the estimate above
and the continuous dependence of initial data on $X^{0,\alpha,r}(T)$
imply that the mild solution depend continuously on initial datum
also in $X^{s,\alpha,r}(T)$.


\subsection{Proof of Theorem \ref{Theorem:2.7}}
Under the assumptions of Theorem \ref{Theorem:2.7},
there exists $r \geq 1$ satisfying  $r> \frac{2n}{2\alpha + n}$ that
	\[
	\rho \geq 1 + \frac{r\sigma}{n},\quad
	\rho > 1 + \frac{\sigma}{n}.
	\]
Therefore,
Theorem \ref{Theorem:2.7} follows from Lemma \ref{Lemma:5.1}
and Corollary \ref{Corollary:6.6}. The time decay estimates (\ref{eq:2.9})-(\ref{eq:2.11}) follows from the estimate $\|u\|_{X^{s,\alpha,r}(\infty)}\le C\varepsilon$ for some positive constant independent of $\varepsilon$ and the definition of the $X^{s,\alpha,r}(\infty)$-norm.

\appendix

\section{Proof of a nonlinear estimate}
\begin{Lemma}
\label{Lemma:A.1}
Let $z=(z_j)_{j=-1}^1, w=(w_j)_{j=-1}^1 \subset \mathbb R$.
For any real sequence $(a_j)_{j=-1}^1$,
let $\Delta^{(2)}(a) = a_{1}+a_{-1} - 2 a_{0}$.
Then
	\begin{align*}
	&| \Delta^{(2)}(\mathcal M (z)) - \Delta^{(2)}(\mathcal M (w))|\\
	&\leq C \max_j |z_j|^{\rho_0} | \Delta^{(2)} (z-w)|\\
	& + C \max_j (|z_j|+|w_j|)^{(\rho_0-2)_+} \max_j |z_j-w_j|^{\min(\rho_0-1,1)} | \Delta^{(2)} (w)|\\
	& + C \max_j |z_j|^{(\rho_0-2)_+}
		(|z_1-z_0| + |z_0 - z_{-1}|)^{\min(\rho_0-1,1)} |z_0 - w_0 - z_{-1} + w_{-1}|\\
	& + C \max_j (|z_j|+|w_j|)^{(\rho_0-2)_+} |w_0 - w_{-1}|\\
	&\cdot
	\min(
		(|z_1-z_0| + |z_0 - z_{-1}| + |w_1-w_0| + |w_0 - w_{-1}|),
		\max_j|z_j - w_j|
	)^{\min(\rho_0-1,1)}
	\end{align*}
with $\mathcal M$ of Lemma \ref{Lemma:6.3}.
\end{Lemma}

\begin{proof}
Let
	\[
	z_\theta
	= \begin{cases}
	\theta z_1 + (1-\theta) z_0
	&\mathrm{if} \quad \theta \in [0,1],\\
	|\theta| z_{-1} + (1-|\theta|) z_0
	&\mathrm{if} \quad \theta \in [-1,0].
	\end{cases}
	\]
Then
	\begin{align*}
	&\Delta^{(2)}(\mathcal M(z))\\
	&= \int_0^1 \mathcal M' (z_\theta) d \theta (z_1-z_0)
	- \int_{-1}^0 \mathcal M' (z_\theta) d \theta (z_0-z_{-1})\\
	&= \int_0^1 \mathcal M' (z_\theta) d \theta \Delta^{(2)}(z)
	+ \int_0^{1} (\mathcal M'(z_\theta) - \mathcal M' (z_{-\theta})) d \theta
	(z_{0}- z_{-1}).
	\end{align*}
Therefore
	\begin{align*}
	&\Delta^{(2)}(\mathcal M(z) - \mathcal M(w))\\
	&= \int_0^1 \mathcal M' (z_\theta) d \theta \Delta^{(2)}(z-w)
	+ \int_0^1 (\mathcal M' (z_\theta) - \mathcal M'(w_\theta)) d \theta \Delta^{(2)}(w)\\
	&+ \int_0^{1} (\mathcal M'(z_\theta) - \mathcal M' (z_{-\theta})) d \theta
	(z_{0} - w_0 - z_{-1} + w_{-1})\\
	&+ \int_0^{1}
		(\mathcal M'(z_\theta) - \mathcal M' (z_{-\theta})
		- \mathcal M'(w_\theta) + \mathcal M' (w_{-\theta})
		) d \theta
	(w_0 - w_{-1}).
	\end{align*}
Since $z_\theta - z_{-\theta} = \theta (z_1 - z_{-1})$,
the assertion is obtained.
\end{proof}

\section{Proof of Lemma \ref{Lemma:4.2}}
\label{section:B}

\begin{proof}[Proof of Lemma \ref{Lemma:4.2}]
When $s=0$,
the bilinear estimate is easily seen with \eqref{eq:1.15}.
When $s>0$, it is sufficient for us to show
	\begin{align}
	\| \phi_j \ast f \sum_{k \leq j-2} \phi_k \ast g \|_{L^{p_0}}
	&\leq C \| \phi_j \ast f \|_{L^{p_1}} \| g \|_{L^{p_2}},
	\label{eq:B.1}\\
	\bigg\| 2^{sj} \| \sum_{k \geq 0} (\phi_{j+k} \ast f) (\phi_{j+k} \ast g) \|_{L^{p_0}} \bigg\|_{\ell^2}
	&\leq C \| f \|_{\dot B_{p_1,2}^s} \| g \|_{L^{p_2}}.
	\label{eq:B.2}
	\end{align}
\eqref{eq:B.1} follows from the Young estimate and the fact that
	\[
	\|\sum_{k \leq j-2} \phi_k \|_{L^1}
	= \|\sum_{k \leq 0} \phi_k \|_{L^1}
	\]
holds for any $j \in \mathbb Z$.
\eqref{eq:B.2} also holds because we have
	\begin{align*}
	\bigg\| 2^{sj} \| \sum_{k \geq 0} (\phi_{j+k} \ast f) (\phi_{j+k} \ast g) \|_{L^{p_0}}
	\bigg\|_{\ell_j^2}
	&\leq
	\| g \|_{L^{p_2}} \sum_{k \geq 0} 2^{-sk} \| 2^{s(j+k)} \|\phi_{j+k} \ast f \|_{L^{p_1}} \|_{\ell_j^2}\\
	&\leq C
	\| g \|_{L^{p_2}} \| f \|_{\dot B_{p_1,2}^s}.
	\end{align*}
\end{proof}


\section*{Acknowledgement}
The first author is supported by Grant-in-Aid for JSPS Fellows 19J00334.
The second author is supported by JST CREST Grant Number JPMJCR1913, Japan.
The second and third authors have been partially supported by
the Grant-in-Aid for Scientific Research (B) (No.18H01132)
and Young Scientists Research (No.19K14581), (No. JP16K17625) Japan Society for the Promotion of Science.




\begin{thebibliography}{00}

\bibitem{INT}
{\sc J. Bergh and J. L\"ofstr\"om},
Interpolation spaces, Springer, Berlin/Heiderberg/New York
(1976).

\bibitem{ChdLuIk13}
{\sc  R.C. Char\~{a}o, C.R. da Luz, and R. Ikehata},
{\em Sharp decay rates for wave equations with a fractional damping via new method in the Fourier space},
J.\ Math.\ Anal.\ Appl.\ {\bf 408} (2013), 247--255.


\bibitem{CFZ14}
{\sc J. Chen, D. Fan, and C. Zhang},
{\em Space-Time Estimates on Damped Fractional Wave Equation},
Abstr.\ Appl.\ Anal.\
{\bf 2014}(2014),
Article ID 428909.


\bibitem{CFZ15}
{\sc J. Chen, D. Fan, and C. Zhang},
{\em Estimates for damped fractional wave equations and applications},
Electronic Journal of Differential Equations {\bf 2015} (2015), 1--14.


\bibitem{ChHa03}
{\sc R. Chill and A. Haraux},
{\em An optimal estimate for the difference of solutions of two abstract evolution equations},
J. Differential Equations {\bf 193} (2003), 385--395.

\bibitem{DaEb14NA}
{\sc M. D'Abbicco and M.R. Ebert},
 {\em An application of $L^p-L^q$ decay estimates to the semi-linear wave equation
 with parabolic-like structural damping},
 Nonlinear Anal.\ {\bf 99} (2014), 16--34.

\bibitem{DaEb14JDE}
{\sc M. D'Abbicco and M.R. Ebert},
{\em Diffusion phenomena for the wave equation with structural damping in the $L^p-L^q$ framework},
J.\ Differential Equations {\bf 256} (2014), 2307--2336.

\bibitem{DaEb16}
{\sc M. D'Abbicco and M.R. Ebert},
{\em A classification of structural dissipations for evolution operators},
Math.\ Meth.\ Appl.\ Sci.\ {\bf 39} (2016), 2558--2582.


\bibitem{DaRe14}
{\sc M. D'Abbicco and M. Reissig},
{\em Semilinear structural damped waves},
 Math.\ Methods Appl.\ Sci.\ {\bf 37} (2014), 1570--1592.

\bibitem{dLuChIk15}
{\sc C.R. da Luz, R. Ikehata, and R.C. Char\~{a}o},
{\em Asymptotic behavior for abstract evolution differential equations of second order},
J. Differential Equations {\bf 259} (2015), 5017--5039.


\bibitem{DaoRe19}
{\sc T.A. Dao, M. Reissig},
{\em An application of $L^1$ estimates for oscillating
integrals to parabolic like semi-linear
structurally damped $\sigma$-evolution models},
J.\ Math.\ Anal.\ Appl.\ {\bf 476} (2019), 426--463.

\bibitem{DaoRe19DCDS}
{\sc T.A. Dao, M. Reissig},
{\em $L^1$ estimates for oscillating integrals and
their applications to semi-linear models
with $\sigma$-evolution like structural damping},
Discrete Contin.\ Dyn.\ Syst.\ {\bf 39} (2019), 5431--5463.

\bibitem{FaLuRe10}{\sc D. Fang, X. Lu, M. Reissig},
{\em High-order energy decay for structural damped systems in the electromagnetical field},
Chin.\ Ann.\ Math.\ Ser. B, {\bf 31} (2010),
237--246.

\bibitem{F} {\sc H. Fujita},
{\em On the blowing up of solutions of the Cauchy problem for
$u_t = \Delta u + u^{1+\alpha}$},
J.\ Fac.\ Sci.\ Univ.\ Tokyo Sect.\ I,
{\bf 13} (1966),
109--124.

\bibitem{GOV}
J. Ginibre, T. Ozawa, and G. Velo,
{\em On the existence of the wave operators for a class of nonlinear Schr\"odinger equations}
Annales de l’I. H. P., Physique th\'eorique,
\textbf{60}(1994), no.2, 211--239.

\bibitem{Gol}
{\sc S. Goldstein},
 {\em On diffusion by discontinuous movements and the telegraph equation},
 Q. J. Mech. Appl. Math. {\bf 4} (1951), 129--156.


\bibitem{HMOW}
H. Hajaiej, L. Molinet, T. Ozawa, and B. Wang,
{\em Necessary and sufficient conditions for the fractional Gagliardo-Nirenberg inequalities and applications to Navier-Stokes and generalized boson equations. Harmonic analysis and nonlinear partial differential equations},
RIMS Kokyuroku Bessatsu,
\textbf{B26}(2011)
159--175.

\bibitem{HKN04}{\sc N. Hayashi, E. I. Kaikina and P. I. Naumkin},
{\em Damped wave equation with super critical nonlinearities},
Diff. Integral Equ., {\bf 17} (2004), 637--652.

\bibitem{HoOg04}{\sc T. Hosono, T. Ogawa},
{\em Large time behavior and $L^p$-$L^q$ estimate of
solutions of 2-dimensional nonlinear damped wave equations},
J. Differential Equations {\bf 203} (2004), 82--118.

\bibitem{IIOW}
M. Ikeda, T. Inui, M. Okamoto and Y. Wakasugi, $L^p-L^q$ estimates for the damped wave equation and the critical exponent for the nonlinear problem with slowly decaying data, Commun. Pure. Appl. Anal., \textbf{18} (2019), 


\bibitem{IIW}
M. Ikeda, T. Inui, and Y. Wakasugi,
The Cauchy problem for the nonlinear damped wave equation with slowly decaying date,
Nonlinear Differ. Equ. Appl.,
(2017) 24:10.

\bibitem{IkNa12}
{\sc R. Ikehata, M. Natsume},
 {\em Energy decay estimates for wave equations with a fractional damping},
 Differential Integral Equations {\bf 25} (2012), 939--956.

\bibitem{IkNi03}
{\sc R. Ikehata, K. Nishihara},
{\em Diffusion phenomenon for second order linear evolution equations},
Studia Math.\ {\bf 158} (2003), 153--161.


\bibitem{IkToYo13JDE}
{\sc R. Ikehata, G. Todorova, B. Yordanov},
{\em Wave equations with strong damping in Hilbert spaces}, 
J. Differential Equations {\bf 254} (2013), 3352--3368.

\bibitem{K79}{\sc M. Kac},
{\em A stochastic model related to the telegrapher's equation},
J. Math. {\bf 4} (1974), 497--509.

\bibitem{KaRe18}{\sc A. Kainane, M. Reissig},
{\em Semi-linear fractional $\sigma$-evolution equations with mass or power non-linearity},
NoDEA Nonlinear Differential Equations Appl.\ {\bf 25} (2018), Art. 42, 43 pp.


\bibitem{KaRe15ADE}
{\sc M. Kainane, M. Reissig},
{\em Qualitative properties of solution to structurally damped $\sigma$-evolution models
with time increasing coefficient in the dissipation},
Adv.\ Differential Equations {\bf 20} (2015), 433--462.

\bibitem{Kar00}
{\sc G. Karch},
{\em Selfsimilar profiles in large time asymptotics of solutions to damped wave equations},
Studia Math.\ {\bf 143} (2000), 175--197.

\bibitem{LZ95}{\sc T.-T. Li, Y. Zhou},
{\em Breakdown of solutions to $\Box u+u_t=|u|^{1+\alpha}$},
Discrete Contin. Dyn. Syst {\bf 1} (1995), 503--520.

\bibitem{LuRe09}
{\sc X. Lu, M. Reissig},
{\em Rates of decay for structural damped models with decreasing in time coefficients},
Int.\ J.\ Dynamical Systems and Diff.\ Eqns. {\bf 2} (2009), 21--55.

\bibitem{Ma76}{\sc A. Matsumura},
{\em On the asymptotic behavior of solutions of semi-linear wave equations},
Publ.\ Res.\ Inst.\ Math.\ Sci.\ {\bf 12} (1976), 169--189.

\bibitem{Na04}
{\sc T. Narazaki},
{\em $L^p$-$L^q$ estimates for damped wave equations
and their applications to semi-linear problem},
J. Math.\ Soc.\ Japan {\bf 56} (2004), 585--626.

\bibitem{Ni03MathZ}{\sc K. Nishihara},
{\em $L^p-L^q$ estimates of solutions to the damped wave equation
in 3-dimensional space and their application},
Math.\ Z. {\bf 244} (2003), 631--649.


\bibitem{PhKaRe15}{\sc D.T. Pham, M. Kainane, M. Reissig},
{\em Global existence for semi-linear structurally damped $\sigma$-evolution models},
J.\ Math.\ Anal.\ Appl.\ {\bf 431} (2015), 569--596.


\bibitem{RaToYo11}{\sc P. Radu, G. Todorova, B. Yordanov},
{\em Diffusion phenomenon in Hilbert spaces and applications},
J. Differential Equations {\bf 250} (2011), 4200--4218. 


\bibitem{SS}
{\sc Y. Shibata and S. Shimizu},
{\em A decay property of the Fourier transform and its application to the Stokes problem},
J. Math. Fluid Mech,
\textbf{3}(2001),
213--320.

\bibitem{S19}
{\sc M. Sobajima}
{\em Global existence of solutions to semilinear damped wave equation with slowly decaying initial data in exterior domain},
Diff. Int. Equs., \textbf{32} (2019),
615--638.

\bibitem{TY01}{\sc G. Todorova, B. Yordanov},
{\em Critical exponent for a nonlinear wave equation with damping},
J. Diff. Equs., {\bf 174} (2001), 464--489. 

\bibitem{V}
{\sc J. L. V\'azquez},
{\em Asymptotic behaviour for the fractional heat equation in the Euclidean space},
Complex Variables and Elliptic Equations
{\bf 63} (2018), 1216--1231.

\bibitem{Wi06}
{\sc J. Wirth},
{\em Wave equations with time-dependent dissipation. I. Non-effective dissipation}, J. Differential Equations, {\bf 222} (2006), 487--514.

\bibitem{Wi07}
{\sc J. Wirth},
{\em Wave equations with time-dependent dissipation. II. Effective dissipation}, J. Differential Equations, {\bf 232} (2007), 74--103.

\bibitem{QZ01}
{\sc Q.S. Zhang},
{\em A blow-up result for a nonlinear wave equation with damping: the critical case}, C. R. Acad. Sci. Paris S\'er. I Math., {\bf 333} (2001), 109--114.

\end{thebibliography}
\end{document}